\documentclass[a4paper,oneside,11pt]{article}
\usepackage{amsmath}
\usepackage{amsfonts}
\usepackage{amssymb}
\usepackage{amsthm}
\usepackage{mathrsfs}
\usepackage{mathtools}
\usepackage{titlesec}
\usepackage{graphicx}
\usepackage{subfigure}
\usepackage{soul}
\usepackage{tocbibind}  % 自动处理目录项添加
\usepackage[thicklines]{cancel}
\usepackage{fancyhdr}
\usepackage{microtype}
\usepackage{float}
\usepackage{latexsym}
\usepackage{wasysym}
\usepackage{cite} 
\usepackage{color}
\usepackage{geometry}
\usepackage{extarrows}
\usepackage{enumerate}
\usepackage{dsfont}
\usepackage{bm}
\usepackage[marginal]{footmisc}
\usepackage[dvips]{epsfig}
\usepackage{amscd}
\usepackage[bookmarks,colorlinks,citecolor=blue,linkcolor=red]{hyperref}
\allowdisplaybreaks
\numberwithin{equation}{section}
\newtheorem{definition}{Definition}[section]
\newtheorem{theorem}{Theorem}[section]
\newtheorem{lemma}{Lemma}[section]
\newtheorem{remark}{Remark}[section]

\renewcommand{\thefootnote}{}

\title{Stability of admissible solutions for coexisting phase transitions for one-dimensional compressible van der Waals fluids}
\date{  }
\author{Yazhou C{\small HEN}$^1$, Qiaolin H{\small E}$^2$, Dongjuan N{\small IU}$^{3}$, Yi P{\small ENG}$^{1*}$, Xiaoding S{\small HI}$^{1}$ \\[3mm]
\scriptsize$^{1}$ {College of Mathematics and Physics, Beijing University of
Chemical Technology, Beijing 100029, China}\\
\scriptsize$^{2}$ {School of Mathematics, Sichuan University, Chengdu 610065,  China}\\
\scriptsize$^{3}$ {School of Mathematical Science, Capital Normal University, Beijing 100048, China }
}

\begin{document}
\maketitle
\renewcommand{\thefootnote}{\fnsymbol{footnote}}
\footnotetext[1]{{Corresponding author. }\\ {Email:  chenyz@mail.buct.edu.cn (Y.Chen), qlhejenny@scu.edu.cn (Q.He), djniu@cnu.edu.cn (D.Niu), apengyi@163.com (Y.Peng), shixd@mail.buct.edu.cn (X.Shi)}}

\begin{abstract}
In this paper, we investigate the dynamic stability of certain steady-state solutions to the periodic boundary value problem for compressible isentropic Navier-Stokes system under the van der Waals equation of state in one space dimension. 
These steady-state solutions correspond to the admissible solutions describing two-phase coexisting phase transitions, where the integral average of the specific volume belongs to the Maxwell region. 
We first construct a semi-discrete staggered grid difference scheme to prove the local existence of solutions to the periodic problem, without imposing the standard stability hypothesis \(p_v<0\). 
Then, by virtue of rigorous piecewise a priori estimates, we demonstrate that the periodic boundary value problem for van der Waals fluids possesses a global solution existing for all time, and this solution converges uniformly to the admissible steady state as time tends to infinity. This result firmly establishes the nonlinear stability of the admissible phase-transition solutions under general small initial disturbances. 

\end{abstract} 

\


\


\section{Introduction}\label{sec:int}
\indent\qquad
The compressible Navier--Stokes equations form the fundamental continuum model describing the motion of viscous,  and stand as one of the central objects in the mathematical theory of nonlinear partial differential equations.
In one space dimension, the system reduces to a parabolic-hyperbolic coupled system, which not only captures essential physical phenomena such as shock waves, contact discontinuities, and phase transitions, but also allows for rigorous mathematical analysis that serves as a foundation for understanding multi-dimensional flows. Periodic boundary conditions model spatially homogeneous, closed flow systems, free of boundary layers and boundary fluxes, providing the cleanest analytical setting for nonlinear PDEs.

In this paper, we investigate the $2L$-periodic boundary value problem for isothermal compressible Navier-Stokes equations of van der Waals fluids in Lagrangian coordinates:
\begin{equation}\label{ns-lagrange}
\left\{\begin{array}{llll}
\displaystyle v_{t}-u_{x}=0,& (x,t)\in\mathbb{R}\times\mathbb{R}_+,\\
\displaystyle u_{t}+p(v)_{x}=\big(\frac{\epsilon u_{x}}{v}\big)_{x},  & (x,t)\in\mathbb{R}\times\mathbb{R}_+,\\
\displaystyle (v,u)(x-L,t)=(v,u)(x+L,t), & (x,t)\in\mathbb{R}\times\mathbb{R}_+,\\
\displaystyle  (v,u)(x,0)=(v_0,u_0)(x), & x\in\mathbb{R},
\end{array}\right.
\end{equation}
where $v(x,t)$ and $u(x,t)$ are unknown functions representing the fluid's specific volume and velocity, respectively, and $\epsilon$ is the viscosity coefficient, which is a positive constant in this paper. $x$ is the Lagrangian coordinate, so that $x=$constant corresponds to a particle path. 
The pressure $p=p(v)$ satisfies the van der Waals equation of state,
\begin{equation}\label{p-f}
\left(p+\frac{a}{v^2}\right)(v-b)=RT.
\end{equation}
Here  $T>0$ represents the temperature of the fluid. $R$ is the gas constant,  $a>0$ and $b>0$ are positive constants characterizing the effect of the molecular cohesive forces and the finite size of the molecules. 
A necessary condition is that the initial data be periodic:
\begin{equation*}
	(v_0,u_0)(x-L)=(v_0,u_0)(x+L),\qquad x\in\mathbb{R}.
\end{equation*} 

A critical physical constraint is that \(p(v)\) is well-defined only for \(v>b\). According to the van der Waals equation \eqref{p-f}, for temperatures in the range $\frac{a}{4bR}< T < T_c\triangleq \frac{8a}{27Rb}$, the corresponding isotherm on the $p$--$v$ diagram exhibits a non-monotonic segment characterized by a local maximum and a minimum (see Figure \ref{fig-p-v}). This ``hump'' signifies a region of mechanical instability where pressure increases with volume. In this region, the system becomes elliptic-hyperbolic mixed type, which breaks the strict hyperbolicity of ideal gas flows and allows the appearance of phase separation and interfaces.
Throughout this work, we operate under the isothermal assumption with $T < T_c$, a necessary condition for phase change. Under this condition, the pressure function 
\begin{equation}\label{p-vandw}
	p(v) = -\frac{a}{v^2}+\frac{RT}{v-b},\quad v>b,
\end{equation}
given by \eqref{p-f} exhibits the characteristic form shown in Figure \ref{fig-p-v} below:
\begin{enumerate}
\renewcommand{\labelenumi}{\roman{enumi})}
	\item $\displaystyle \lim_{v \to +\infty} p(v) = 0^+$.
	\item $p(v)$ is non-monotone and possesses exactly two critical points $\beta > \alpha > b$ such that $p'(\alpha) = p'(\beta) = 0$ and $p(v)$ is strictly decreasing on $(b, \alpha)$ and $(\beta, +\infty)$, and strictly increasing on $(\alpha, \beta)$.	
	\item There exist unique points $\alpha_0 \in (b, \alpha)$ and $\beta_0 \in (\beta, \infty)$ such that the Maxwell construction (i.e. Maxwell equal-area rule) \cite{M1875} holds:
	\begin{equation}\label{Maxwell equal area rule}
		\int_{\alpha_0}^{\beta_0} \big( p(v) - p(\alpha_0) \big) \, dv = 0,\quad p(\alpha_0)=p(\beta_0).
	\end{equation}
	From a physical perspective, the interval  $(\alpha_0,\beta_0)$ is known as the coexistence region, commonly referred to as the Maxwell region, and is denoted by
	\begin{equation}\label{M-R}
		\Omega_{\mathrm{Maxwell}} =(\alpha_0,\beta_0).
	\end{equation}
\end{enumerate}
\begin{figure}[htbp!]
\centering
\includegraphics[width=0.5\textwidth]{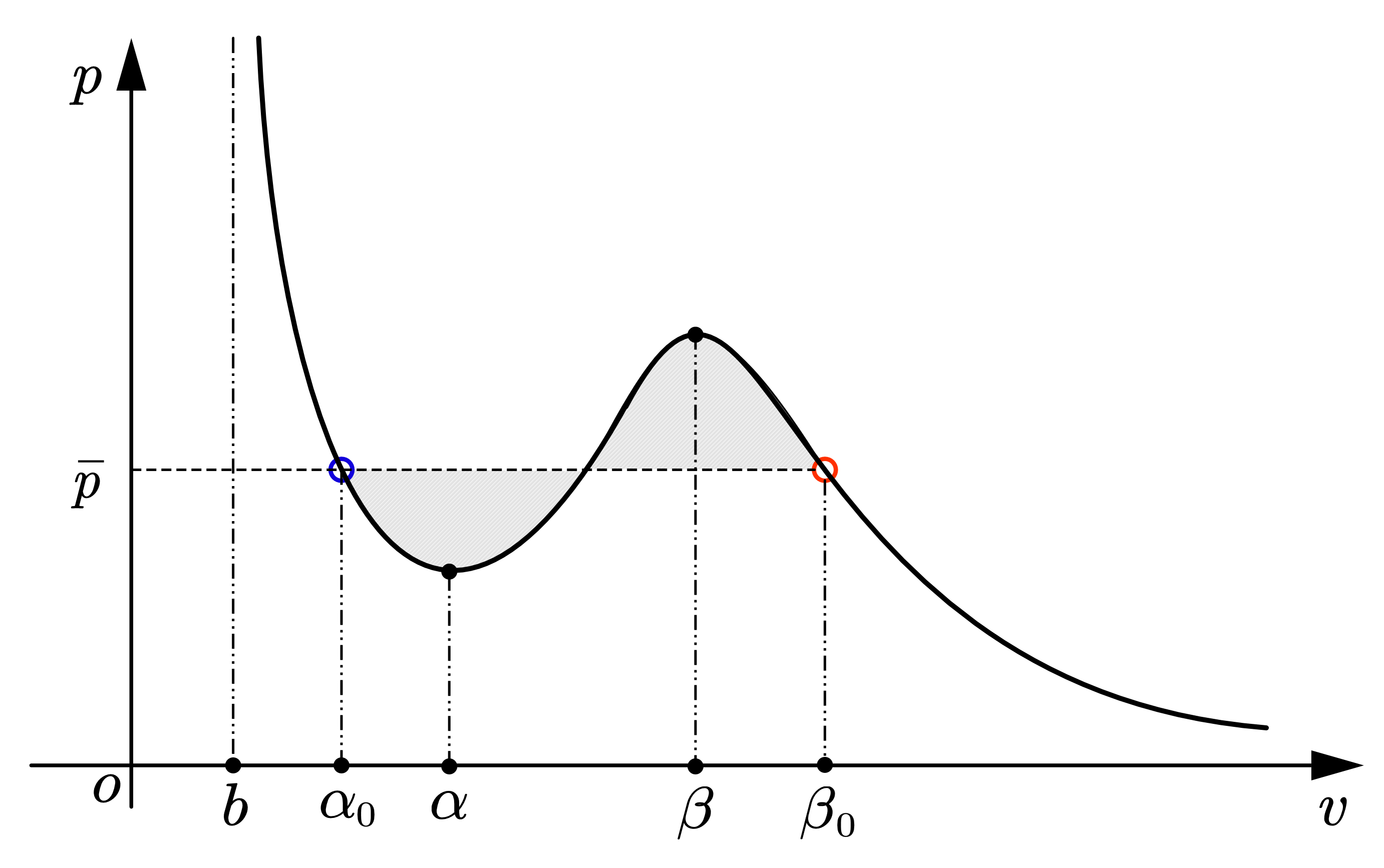}
\caption{Figure of $p-v$}
\label{fig-p-v}
\end{figure}

Numerical results from Hsieh-Wang \cite{HW1997} indicate that the interval $(b, +\infty)$ admits a decomposition into three regions: $\Omega_{\mathrm{unstable}}$, $\Omega_{\mathrm{metastable}}$, and $\Omega_{\mathrm{stable}}$ (see Figure \ref{fig-p-v}):
\begin{equation}\label{ABC}
\left.\begin{array}{llll}
\displaystyle \Omega_{\mathrm{unstable}}\triangleq(\alpha,\beta), & \text{unstable region},\\
\displaystyle \Omega_{\mathrm{metastable}}\triangleq(\alpha_0,\alpha]\cup[\beta,\beta_0),& \text{metastable region},\\
\displaystyle \Omega_{\mathrm{stable}}\triangleq(b,\alpha_0]\cup[\beta_0,+\infty),& \text{stable region}.
\end{array}\right.
\end{equation}
Based on the definition \eqref{ABC}, it immediately follows that the Maxwell region is exactly the union of the unstable and metastable regions, i.e.,
\begin{equation}\label{Maxwell region and unstable metastable}
  \Omega_{\mathrm{Maxwell}}=(\alpha_0,\beta_0)=\Omega_{\mathrm{unstable}}\cup \Omega_{\mathrm{metastable}}.
\end{equation}

Given that the periodic solutions $(v, u)(x, t)$ of the system \eqref{ns-lagrange} in 
$\mathbb{R}$ can be interpreted as $2L$-periodic extensions of their restrictions to \(\mathbb{T}=[-L,L]\), we hereafter restrict our analysis of the periodic problem \eqref{ns-lagrange} to the bounded interval $\mathbb{T}$. By integrating \eqref{ns-lagrange}$_1$ and  \eqref{ns-lagrange}$_2$ over the domain $\mathbb{T}\times[0,t]$ respectively,  and applying the periodic boundary condition specified in \eqref{ns-lagrange}$_3$, we derive the following requirement:
\begin{equation}\label{Hypo-0}
	\int_{\mathbb{T}}v(x,t)dx=\int_{\mathbb{T}}v_0(x)dx,\quad \int_{\mathbb{T}}u(x,t)dx=\int_{\mathbb{T}}u_0(x)dx.
\end{equation}
We define the integrated average of specific volume and velocity over the period $\mathbb{T}$ by
\begin{equation}\label{Hypo-1}
	\bar{v}\triangleq\frac{1}{2L}\int_{\mathbb{T}}v_0(x)dx,\quad
	\bar{u}\triangleq\frac{1}{2L}\int_{\mathbb{T}}u_0(x)dx.
\end{equation}

Suppose the initial data $(v_0, u_0)$ has a finite number of jump discontinuities, say at $x=y_i,~j=1,2,\dots,N_0$.
We now describe the jump conditions alluded to above. Suppose that \((v, u)\) is a solution to \eqref{ns-lagrange}, and which is smooth except on a curve \(x(t)\) whose speed is \(\dot{x}=s\), 
One easily checks that $(v, u)$ satisfies the Rankine-Hugoniot conditions:
\begin{equation}\label{rh-conditon1}
	-s[\![v]\!]=[\![u]\!], \quad s[\![u]\!]=[\![p - \frac{\epsilon u_{x}}{v}]\!],
\end{equation}
where $[\![\cdot]\!]$ denotes the jump in a quantity across the discontinuity curve. For example, $[\![v]\!] = v(x+, t) - v(x-, t)$. 
Now, the result of Hoff-Smoller \cite[Theorem 4.2]{hs1985} shows that the initial discontinuities in $u$ must be smoothed out in ${t >0}$, i.e.
$[\![u]\!] = 0$. It then follows that $s = 0$ and
\begin{equation}\label{jump0}
	[\![p(v) - \frac{\epsilon u_{x}}{v}]\!]=0.
\end{equation}
We thus expect that a discontinuity in $v_0$ at $x = y_i$ propagates in the solution along the particle path $x = y_i$.
% For functions \(g=g(w)\) we use the standard divided difference notation
% \begin{equation}\label{g-dd}
% 	g_{w_{1},w_{2}}=\begin{cases} \displaystyle \frac{g(w_{2})-g(w_{1})}{w_{2}-w_{1}},&\quad w_{1}\neq w_{2},\\ g'(w_{1}),&\quad w_{1}=w_{2}.\end{cases}
% \end{equation}
%Now suppose as above that the jump condition \eqref{jump0} hold across \(x = y_i\). 
Letting \(L = \ln v\), we then find from \eqref{ns-lagrange} and \eqref{jump0} that
\begin{equation}\label{eq:jump-alpha0}
{[\![   L]\!]  }_{t} = [\![  \frac{{u}_{x}}{v}]\!]   = \frac{1}{\epsilon }[\![  {p( {v}) }]\!]   = \frac{\alpha_i(t)}{\epsilon }[\![  L]\!],
\end{equation}
where 
\begin{equation}\label{eq:jump-alpha1}
 \alpha_i(t)=%\frac{p_{v-,v+}}{L_{v-,v+}}=
 \frac{p(v+)-p(v-)}{L(v+)-L(v-)}=\xi p'(\xi),
\end{equation}
and \({v}_{ \pm  } = v( {y_i\pm ,t}) \), $\xi$ is between $v+$ and $v-$.
Solving \eqref{eq:jump-alpha0} we then conclude that
\begin{equation}\label{eq:jump-alpha2}
[\![  {L( {y_i,t}) }]\!]   = [\![  {L( {y_i,0}) }]\!]  \exp \Big( {{\epsilon }^{-1}\int_{0}^{t}\alpha_i (s) {ds}}\Big).
\end{equation}
Given any small positive number $\underline{v}>0$ and $\hat{v}>\beta_0$, for $v\in [b+\underline{v}, \hat{v}]$, it is easy to get that
\begin{equation}\label{eq:vvv}
\begin{aligned}
\displaystyle & -C^{-1}\leq p'(v)\leq C,\quad %\\& 
-C^{-1}\leq vp'(v)\leq C,
\end{aligned}	
\end{equation}
where $C$ is a positive constant depending only on $a, b, \underline{v}$ and $\hat v$.
Therefore  \eqref{eq:jump-alpha2} shows that the magnitude of the discontinuity in \(L(v)\) can grow at most exponentially in time. 
To be precise, %if $v_\pm \in \Omega_{\mathrm{stable}}$, the jump discontinuity of $L(v)$ decays exponentially with time; 
if $v_\pm \in \Omega_{\mathrm{unstable}}$, the amplitude of the discontinuity in $L(v)$ exhibits at most exponential growth over time. In particular, whenever $p(v_+) = p(v_-)$, the magnitude of the discontinuity in $L(v)$ remains constant, yielding a stationary jump that corresponds exactly to the phase transition interface.

Let us now review the results concerning the well-posedness of solution for one-dimensional isentropic (or isothermal) Navier-Stokes equations \eqref{ns-lagrange}. For ideal isentropic flow, seminal contributions include global existence of solutions for Cauchy problem with smooth initial data by Kanel \cite{k1968}, and with nonsmooth initial data by Hoff \cite{hoff1986,hoff1987}, where solutions are obtained as limits of approximations obtained by building heuristic jump conditions into a semi-discrete difference scheme. For the one-dimensional non-isentropic flow governed by the ideal gas equation of state $pv = RT$, the global existence of solutions without smallness assumptions on the smooth initial data was established by Kanel \cite{K1979} and Kazhikhov \cite{K1982} for Cauchy problem and by Kazhikhov and Shelukhin \cite{ks1977} for initial-boundary value problem. 
Further related developments can be found in Jiang \cite{J1999} and  Li-Liang \cite{LL2016} among others. Regarding the asymptotic behavior of solutions, seminal contributions include the stability of rarefaction and shock waves established by Matsumura-Nishihara \cite{MN1985, MN1986}, and the stability of contact discontinuities proved by Huang-Matsumura-Shi \cite{HMS2004} and Huang-Li-Matsumura \cite{HLM2010}. Furthermore, the stability and viscosity limit of interacting shock waves were investigated by Huang-Wang-Wang-Yang \cite{HWWY2015} and Shi-Yong-Zhang \cite{SYZ2016}, respectively.

In the context of van der Waals fluid, Affouf-Caflisch \cite{AC1991} performed numerical studies on phase-transition solutions to the Cauchy problem of isothermal fluid systems equipped with viscosity and capillarity.  
Hsieh-Wang \cite{HW1997} adopted a pseudo-spectral method with artificial viscosity and numerically demonstrated that the non-monotonicity of pressure can induce phase transitions. Subsequently, He-Liu-Shi \cite{HLS2018} revisited this problem by employing a second-order TVD Runge-Kutta splitting scheme coupled with the Jin–Xin relaxation method. Furthermore, this approach was further extended to the compressible Navier–Stokes/Cahn–Hilliard and Navier–Stokes/Allen–Cahn systems in the studies of He-Shi \cite{hs2020} and \cite{hs2021}, respectively.
Hsiao \cite{HSIAO1990} constructed admissible shock and wave curves for the one-dimensional inviscid Riemann problem and further established a shock-rarefaction wave fan that conforms to prescribed initial data. 
For the non-isentropic nonideal gases, Hoff \cite{hoff1991} constructed local-in-time discontinuous weak solutions of the full Navier-Stokes equations for one-dimensional non-isentropic compressible flow with initial data of bounded variation. They also obtained detailed information concerning jump conditions and the evolution in time of the magnitudes of the jump discontinuities, which is similar as that in \eqref{eq:jump-alpha2} for the isothermal flow. %It shows that $[\![  \ln v]\!]$ can grow at most exponentially in time. 
In order to control the jump discontinuities pointwise for all time, Hoff \cite{hoff1992} assumed that $p$ and $T$ satisfy the conditions of a near ideal gas and the quantity $\alpha_i$ in \eqref{eq:jump-alpha2} satisfies $\alpha_i\leq -1/C$, and proved the global existence, uniqueness and continuous dependence on initial data for discontinuous solutions of the one-dimensional full Navier-Stokes equations. 
Later, Hoff-Khodja \cite{hk1993} proved the dynamic stability of certain steady-state solutions of Cauchy problem for compressible non-isothermal van der Waals fluids. These steady-state solutions consist of two constant states, corresponding to different phases, separated by a convecting phase boundary. They also require that the fluid be near-ideal in a neighborhood of each of the constant states, similarly as in \cite{hoff1992}. 
In addition, relevant studies have investigated the well-posedness of solutions for capillary fluid equations with Korteweg terms by Eden-Milani-Nicolaenko \cite{EMN1993}, viscous-capillary wave systems by Zumbrun \cite{Z2000}, and non-monotone dynamic elastic models by Mei-Wong-Liu \cite{MLW2007,MLW2007-2}.
%Eden-Milani-Nicolaenko \cite{EMN1993} applied semigroup methods to investigate the periodic van der Waals Navier–Stokes equations with Korteweg capillary terms. Separately, Zumbrun \cite{Z2000} analyzed the Cauchy problem for the viscous–capillary p-system and derived the linear orbital stability of its traveling wavefront solutions. For dynamic elastic bar theory characterized by a similarly non-monotone stress–deformation relation, Mei-Wong-Liu \cite{MLW2007, MLW2007-2} proved the existence and convergence of global strong solutions to the periodic boundary value problem for one-dimensional Navier–Stokes system augmented with artificial viscosity terms.

The main purpose of this paper is to investigate the asymptotic stability of a class of admissible steady-state solutions related to liquid-vapor phase transitions for the periodic boundary value problem of one-dimensional isothermal compressible van der Waals fluids. 
It is known that the steady-state solution may describe the long-time asymptotic behavior of the unsteady flow 
%Without loss of generality, we assume that the viscosity coefficient $\epsilon=1$ 
, and the steady-state problem of system \eqref{ns-lagrange} under constraint \eqref{Hypo-0} reduces to the following system of ordinary differential equations:
\begin{equation}\label{steady-state ns-lagrange}
\left\{\begin{array}{llll}
\displaystyle -\tilde{u}_{x}=0,& x\in \mathbb{T},\\
\displaystyle p(\tilde{v})_{x}=\big(\frac{\epsilon\tilde{u}_{x}}{\tilde{v}}\big)_{x}, & x\in \mathbb{T}, %\\ \displaystyle  (\tilde{v},\tilde{u})(-L)=(\tilde{v},\tilde{u})(L),%\\\displaystyle (v_x,u_x)(-1)=(v_x,u_x)(1),
\end{array}\right.
\end{equation}
supplemented by the integral constraint
\begin{equation}\label{constraint for v,u}
\frac{1}{2L}\int_{\mathbb{T}}\tilde{v}(x)dx = \bar v.
\end{equation}
Obviously, \(\tilde{u}\) is a constant, without loss of generality, we may set \(\tilde{u}= 0\).
For ideal and near-ideal fluids with $p_v < 0$, the only such steady-state solutions are those for which all dependent variables $\tilde{v}$ and $\tilde{u}$ are constant. Nonconstant steady-state solutions do exist for van der Waals fluids because the isotherms for such fluids need not be monotone. 

Our prior work \cite{chnys2026} proposed an approximation to system \eqref{steady-state ns-lagrange} via artificial viscosity, with the purpose of capturing energy-minimizing equilibrium configurations, the corresponding approximate system is given below,
\begin{equation}\label{Approximate ns-lagrange}
\left\{\begin{array}{llll}
\displaystyle -\tilde{u}_{x}=\varepsilon^2 \tilde{v}\tilde{v}_{xx},& x\in \mathbb{T},\\
\displaystyle p_{x}(\tilde{v})=\big(\frac{\epsilon\tilde{u}_{x}}{\tilde{v}}\big)_{x},& x\in \mathbb{T},
\end{array}\right.
\end{equation}
with constraint \eqref{constraint for v,u}. 
Within an energy framework that accounts for interfacial contributions, it is therefore expected that the minimal-energy two-phase periodic solution will exhibit a double-interface structure. This expectation is confirmed by numerical studies of the one-dimensional periodic problem \eqref{steady-state ns-lagrange} for a van der Waals fluid (see Hsieh-Wang \cite{HW1997}
and references therein). When the average initial density $\bar v$ lies in the Maxwell region, i.e.,
\begin{equation}\label{eq:thm-condition-add}
	\bar v\in(\alpha_0,\beta_0),
\end{equation}
we prove that the approximate solution converges, as the artificial viscosity tends to zero, to the equilibrium states given by Maxwell’s construction, with the diffuse interface sharpening into a discontinuity (see \cite[Theorem 1.3]{chnys2026}). 
We therefore define this limit as the admissible solution of the steady-state equation \eqref{steady-state ns-lagrange}, that is, the solution with physical significance.
The limit, illustrated in Figure~\ref{fig-vv}, is either single-peak or single-valley, each profile corresponds to a piecewise smooth solution of system \eqref{steady-state ns-lagrange}. 
\begin{figure}[htbp!]
\centering
\includegraphics[width=0.8\textwidth]{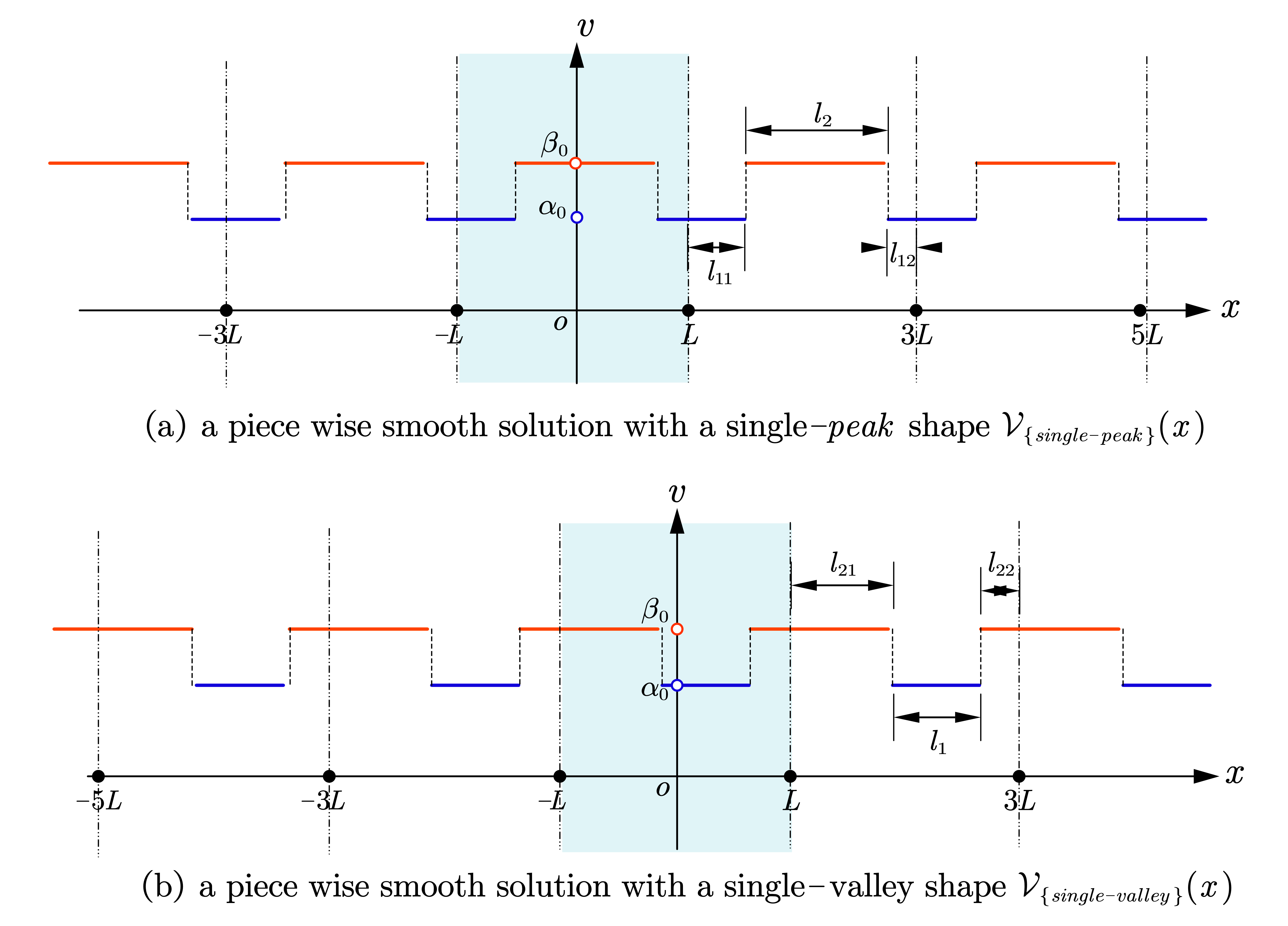}
\caption{phase-transition solution with two distinct interfaces}
\label{fig-vv}
\end{figure}

We consider the admissible steady-state solutions known as the phase-transition solution with two distinct interfaces in \cite{chnys2026}, which are defined precisely as follows.
\begin{definition}\label{phase transition solution with two distinct interfaces}
 A function whose profile is either single-peaked or single-dipped and which constitutes a piecewise smooth solution of system \eqref{steady-state ns-lagrange} is called a phase transition solution with two distinct interfaces. It is denoted by $\mathscr{V}_{\mathrm{single\text{-}peak}}(x)$ or $\mathscr{V}_{\mathrm{single\text{-}valley}}(x)$, and defined explicitly as follows (see Figure~\ref{fig-vv}):
\begin{equation}\label{Phase-transition-solution-1}
  \mathscr{V}_{\mathrm{single\text{-}peak}}(x)=\begin{cases}
  \alpha_0, & x\in[-L,-L+l_{11}),\\
  \beta_0, & x\in (-L+l_{11},-L+l_{11}+l_2),\\
  \alpha_0, & x\in(-L+l_{11}+l_2, L],
\end{cases}
\end{equation}
or
\begin{equation}\label{Phase-transition-solution-2}
  \mathscr{V}_{\mathrm{single\text{-}valley}}(x)=\begin{cases}
  \beta_0, & x\in[-L,-L+l_{21}),\\
  \alpha_0, &  x\in (-L+l_{21},-L+l_{21}+l_1),\\
  \beta_0, &  x\in(-L+l_{21}+l_1, L],
\end{cases}
\end{equation}
where the lengths $l_1$ and $l_2$ are given by
\begin{equation}\label{lii}
  l_1 = l_{11}+l_{12}=\frac{2L(\beta_0-\bar{v})}{\beta_0-\alpha_0},\qquad 
  l_2 = l_{21}+l_{22}=\frac{2L(\bar{v}-\alpha_0)}{\beta_0-\alpha_0}.
\end{equation}
\end{definition}

\begin{remark}
According to the definition in \cite{chnys2026} and letting ${(v^\varepsilon_{\text{single-peak}}, u^\varepsilon)}_{\varepsilon>0}$ (or ${(v^\varepsilon_{\text{single-valley}}, u^\varepsilon)}_{\varepsilon>0}$) be a family of classical solutions to the approximate system \eqref{Approximate ns-lagrange} on $\mathbb{R}$. If ${(v^\varepsilon_{\text{single-peak}}, u^\varepsilon)}_{\varepsilon>0}$ (or ${(v^\varepsilon_{\text{single-valley}}, u^\varepsilon)}_{\varepsilon>0}$) converges almost everywhere to a pair $(v, u)$ as $\varepsilon \to 0^+$, then the limit $(v, u)$ is called an admissible solution of the steady-state system \eqref{steady-state ns-lagrange}. It is shown in \cite{chnys2026} that the limit pair \((v,u)\) corresponds to a phase-transition solution possessing two distinct interfaces denoted by $\mathscr{V}_{\mathrm{single\text{-}peak}}(x)$ or $\mathscr{V}_{\mathrm{single\text{-}valley}}(x)$ in Definition \ref{phase transition solution with two distinct interfaces}. 
This admissible solution possesses a double-interface structure characteristic of coexisting liquid-vapor phase transitions.
\end{remark}

\begin{remark}
While the measures of the two-phase regions $l_1$ and $l_2$ can be determined, the precise locations $-L+l_{11}$ or $-L+l_{21}$ at which the phase transition occurs cannot be identified with the techniques employed. However, it is worth pointing out that by performing a coordinate translation on $x$, the phase transition solutions with two distinct interfaces $\mathscr{V}_{\mathrm{single\text{-}peak}}(x)$ and $\mathscr{V}_{\mathrm{single\text{-}valley}}(x)$ can be transformed into each other. Therefore, it only needs to investigate the stability of the steady-state solution $\mathscr{V}_{\mathrm{single\text{-}peak}}(x)$. 
\end{remark}

In addition, this paper is devoted to the existence of weak solutions to the aforementioned periodic boundary value problem, where the initial specific volume is assumed to possess finitely many jump discontinuities, and the pressure is allowed to cover the non-monotone region. 
The study of weak solutions to the compressible Navier–Stokes equations with a van der Waals-type pressure law is both physically motivated and mathematically essential, particularly in the context of phase transitions. Unlike classical ideal gas models, the van der Waals equation captures the coexistence of liquid and vapor phases, leading to the physically realistic expectation that the specific volume will develop jump discontinuities across phase boundaries. 
Weak solutions are defined in the obvious way as follows.
\begin{definition}\label{def:weak}
	A pair \((v,u)\) is a weak solution to the system \eqref{ns-lagrange} if $v, u, p(v)$ and $\dfrac{\epsilon u_x}{v}$ are locally integrable, and the weak equations 
	\[
	\int_\mathbb{T} v(x,t)\phi(x,t)dx-\int_\mathbb{T} v_0\phi(\cdot,0)dx-\int_0^t\int_\mathbb{T} \left( v\phi_t - u\phi_x \right)dx dt = 0,
	\]
	and
	\[
	\int_\mathbb{T} u(x,t)\psi(x,t)dx- \int_\mathbb{T} u_0\psi(\cdot,0)dx -\int_0^t\int_\mathbb{T} \left( u\psi_t +p(v)\psi_x - \frac{\epsilon u_x}{v}\psi_x \right)dx dt  = 0,
	\]
hold for all test functions $\phi,\psi\in C^{1}\bigl(\,[0,\infty);\,C_{c}^{1}(\mathbb{T})\,\bigr).$
\end{definition}

Without loss of generality, by an appropriate coordinate translation, we make $\mathscr{V}_{\mathrm{single\text{-}peak}}(x)$ take the following form, see Figure~\ref{fig-vv2} (denoting $\mathscr{V}_{\mathrm{single\text{-}peak}}(x)$ simply as $\tilde{v}(x)$):
\begin{equation}\label{steady-sol}
\tilde{v}(x) = 
\begin{cases}
\alpha_0, & x\in I_1 \triangleq \{x|-L<x<-l\}, \\
\beta_0, &  x\in I_2 \triangleq\{x|-l<x<l\}, \\
\alpha_0, & x\in I_3 \triangleq\{x|~l<x<L\},
\end{cases}
~~~\tilde{u}(x)=0,
\end{equation}
where $l=\dfrac{L(\bar{v}-\alpha_0)}{\beta_0-\alpha_0}$.
Then $\tilde{p}=p({\tilde{v}})=p(\alpha_0)=p(\beta_0)$, so that $(\tilde{v}, \tilde{u})$ satisfies \eqref{ns-lagrange} away from the lines $x = \pm l$. 
\begin{figure}[htbp!]
\centering
\includegraphics[width=0.5\textwidth]{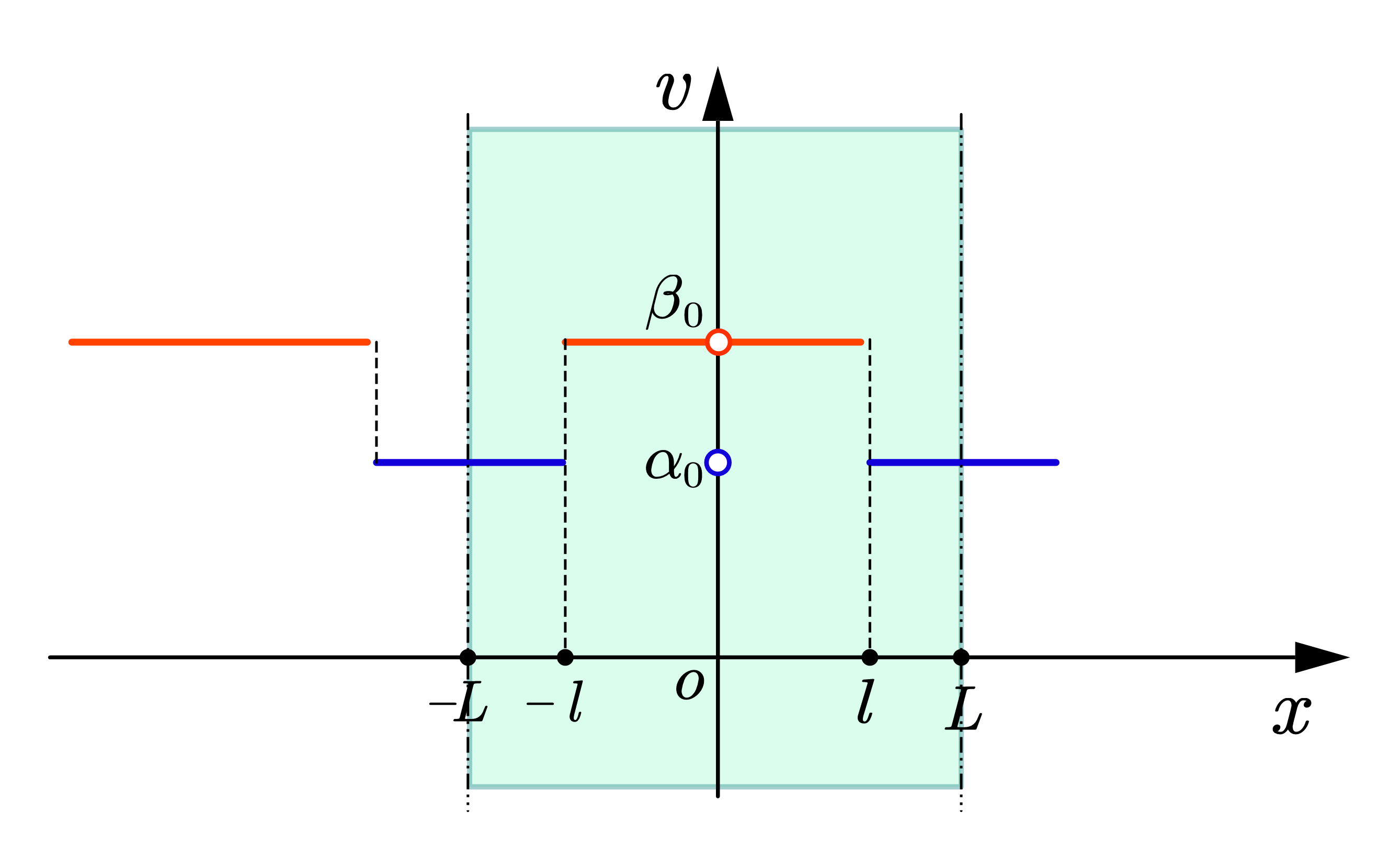}
\caption{steady-state solution with two distinct interfaces}
\label{fig-vv2}
\end{figure}

One easily checks that $(\tilde{v}, \tilde{u})$ also satisfies the Rankine-Hugoniot conditions \eqref{jump0}.
% \begin{equation}\label{jump1}
%   [u] = 0, \quad [p - \frac{u_{x}}{v}] = 0,
% \end{equation}
% for a jump discontinuity convecting with the flow. $[\cdot]$ denotes the magnitude of a jump; for example, $[u] = u(x+, t) - u(x-, t)$. 
The admissible steady-state solution $(\tilde{v}, \tilde{u})$ is then indeed a weak solution of \eqref{ns-lagrange} in the space $\{x|~x\neq l \}$. The lines $x =\pm l$ represent the convecting phase boundaries across which $\tilde{u}, \tilde{p}$ are continuous, but $\tilde{v}$ is discontinuous. Observe that the magnitudes of the jumps in $\tilde{v}$ is not small in any sense and are independent of $t$ and $\epsilon$. This represents a significant qualitative difference from the case of Navier-Stokes flow for near-ideal fluids, in which small convective discontinuities decay to zero exponentially in time, more rapidly for smaller viscosities. 

We now describe our hypothesis on the initial data $(v_0,u_0)$. $(\tilde{v},\tilde{u})$ is the admissible steady-state solution defined in \eqref{steady-sol}. Letting \(\parallel  \cdot  \parallel\) denote the usual \(L^{2}\left( \mathbb{T}\right)\)-norm, we define
\begin{equation}\label{eq:C0}
	C_{0}= \|(v_0-\tilde{v},u_0-\tilde{u})\|^{2}+ \left(\textrm{Var}~u_{0}\right)^{2}+ \left\|v_{0x}\right\|_{L^2_\Sigma}^{2},
\end{equation}
where $\pm l$ are distinguished points at which discontinuities in $v$ occur, and $\|\cdot\|_{L^2_\Sigma}$%\sout{$\Vert$} 
be the piecewise $L^2$ norm defined, for example, by
$$
\| v_x(\cdot, t)\|_{L^2_\Sigma}^2 \triangleq \sum_{i=1}^3\int_{I_i} v_x(x, t)^2 d x, %+ \int_{I_2} v_x(x, t)^2 d x+ \int_{I_3} v_x(x, t)^2 d x.
$$
which is meaningful when the discontinuous variable $v(\cdot, t)$ piecewise smooth function in $x$. 
We require that the fluid be near-ideal in a neighborhood of each of the points $\alpha_0$ and $\beta_0$. Specifically, we define the neighborhood $D_\alpha$ and $D_\beta$ by
\begin{equation}\label{eq:dd}
	D_\alpha=\{v\big|~|v-\alpha_0|<{\hat{\delta}},~v\in (b,\alpha)\},~~
	D_\beta=\{v\big|~|v-\beta_0|<{\hat{\delta}},~v\in (\beta,\hat{v})\},
\end{equation}
where $\hat{\delta}>0$ is a given positive constant. It is easy to check that
\begin{equation}\label{p-assume2}
  p'(v) < 0~~~\text{ in }~D_\alpha \cup D_\beta.
\end{equation}
Next, we fix a number $\theta \in (0,1)$,  let $\sigma = \min\{1, t\}$, and define the following functionals for solutions $(v, u)$ of \eqref{ns-lagrange}:
\begin{align}
	A(t_0) &= \sup_{0 \leq t \leq t_0} \left[\|(v-\tilde{v}, u-\tilde{u})(\cdot, t)\|^2 + \|v_x(\cdot, t)\|_{L^2_\Sigma}^2\right] + \int_0^{t_0} \|u_x(\cdot, t)\|^2d t, \label{eq:A} \\
	B(t_0) &= \sup_{0 < t \leq t_0} \sigma^{1/2} \|u_x(\cdot, t)\|^2 
	+ \int_0^{t_0} \sigma^{1/2+\theta} \left(\|u_t(\cdot, t)\|^2 + \|\left(\frac{u_x}{v}\right)_x(\cdot, t)\|_{L^2_\Sigma}^2\right) d t, \label{eq:B}\\
	%D(t_0) &= \sup_{0 < t \leq t_0} \|u_x(\cdot, t)\|^2	+ \int_0^{t_0} \left(\|u_t(\cdot, t)\|^2 + \left(\frac{u_x}{v}\right)_x(\cdot, t)\|_{L^2_\Sigma}^2\right) d t, \label{eq:D}\\
	F(t_0) &= \sup_{0 < t \leq t_0} \left[\sigma^{3/2} \left(\|u_t(\cdot, t)\|^2 + \|\left(\frac{u_x}{v}\right)_x(\cdot, t)\|_{L^2_\Sigma}^2\right)\right] + \int_0^{t_0} \|u_{xt}(\cdot, t)\|^2  d t. \label{eq:F}
\end{align}
%We also define $D(t_0)$ and $G(t_0)$ to be the same as $B(t_0)$ and $F(t_0)$, but with $\sigma$ taken to be identically one. 
The definitions of these functionals are exactly the same as those in \cite{hoff1991} and \cite{hoff1992}. The only difference is that we remove from $A(t)$ a term that measures the magnitude of the jump discontinuity of the specific volume $v(\cdot, t)$. This term is omitted because the entire analysis in this paper is based on the estimates that show that $A, B,$ and $F$ remain small for all time. In contrast, the jump in $v$ at $x = \pm l$ is not small in any sense.

We can now state the main result of this paper.

\begin{theorem}\label{thm1} Let $C_0, D_\alpha, D_\beta$ be as described above in \eqref{eq:C0}-\eqref{eq:dd}, $(\tilde{v},\tilde{u})$ is the admissible steady-state solution defined in \eqref{steady-sol}.  Given any small positive number $\underline{v}>0$ and $\hat{v}>\beta_0$. Then there are positive constants $\delta_0,\tau$ and $C$ such that, given the initial data $(v_0,u_0)$  satisfying \eqref{Hypo-1}, \eqref{eq:thm-condition-add} and
\begin{align}
	& v_0\in [b+2\underline{v},\hat{v}-\underline{v}],~~a.e., \label{eq:vv-b0}\\
	& \int_{I_i}(v_0-\tilde{v})dx=0,~i=1,2,3,\label{eq:add-b0} \\	
&C_0%= \|(v_0-\tilde{v},u_0-\tilde{u})\|^{2}+ (\textrm{Var}~ u_0)^2+ \left\|(v_{0})_x\right\|_{L^2_\Sigma}^{2}
	\leq \delta_0, \label{eq:e0-b0}
\end{align}	
then the periodic boundary problem \eqref{ns-lagrange} has a global weak solution $(v,u)$ %with the initial data $(v_0,u_0)$
 satisfying
\begin{align}
&v \in  [b+\underline{v},\hat{v}], &&  0\leq t\leq \tau,\label{eq:vv-b11}\\
&v \in 
\begin{cases}
D_\alpha, & x \in I_1 \cup I_3, \\
D_\beta, & x \in I_2,
\end{cases}  && t\geq \tau, \label{eq:vv-b1}\end{align}
\begin{align} \displaystyle A(t)+B(t)+F(t)\leq CC_0. \label{eq:e0-b1}
\end{align}
Moreover, the solution $(v,u)$ tends to $(\tilde{v},\tilde{u})$ as $t\rightarrow\infty$ in the sense that
\begin{equation}\label{large-time}
\lim_{t\rightarrow\infty}\|(v-\tilde{v},u-\tilde{u})(\cdot,t)\|_{L^\infty(\mathbb{T})}=0.
\end{equation}
\end{theorem}

\begin{remark}
Theorem \ref{thm1} shows that for any $\bar v \in \Omega_{\mathrm{Maxwell}}$, the periodic problem for system \eqref{ns-lagrange} admits a global solution converging uniformly to the admissible steady-state solution as $t\to+\infty$.  This admissible solution possesses a double-interface structure characterized by two separate density jumps corresponding to liquid–vapor phase transitions. It is shown that such admissible solution is nonlinearly stable under sufficiently small perturbations of the initial data.
Conversely, if $\bar v$ lies outside the Maxwell region, the solution remains in a single-phase state.
Our results indicate that if the left and right limits of $v$ at a discontinuity coincide with the Maxwell points $\alpha_0$ and $\beta_0$, the amplitude of the jump discontinuity in $v$ remains constant, yielding a stationary jump that corresponds exactly to the phase transition interface. 
These findings demonstrate that the thermodynamic instability within the Maxwell regime serves as the driving force triggering phase separation.
\end{remark}

\begin{remark}\label{rem:1}
The regularity \eqref{eq:e0-b1} is more than sufficient for those integrals to be defined in Definition \ref{def:weak}. Indeed, \eqref{eq:e0-b1} implies certain H\"older regularity for the solution \((v,u)\). Specifically, we employ the standard notation for H\"older norms
\[\langle w\rangle_A^{a,b}=\sup_{(x,t),(y,s)\in A}\frac{|w(x,t)-w(y,s)|}{|x-y|^a+|t-s|^b}.\]
The following regularity results are consequences of \eqref{eq:e0-b1}:
	\begin{align}
		&\langle v\rangle_{I_i\times[0,\infty)}^{\frac{1}{2},\frac{1}{2}-\frac{\theta}{4}}\leq CC_{0}^{\frac{1}{2}},\quad i=1,2,3,\label{eq:thm1-v1}	\\	
		&\sigma^{\frac{1}{4}+\frac{\theta}{2}}\langle u\rangle_{\mathbb{T}\times[t,\infty)}^{\frac{1}{2},\frac{1}{4}} \leq CC_{0}^{\frac12},\label{eq:thm1-u1} \\
		&\sigma^{\frac{1}{2}+\frac{\theta}{4}}\|u_{x}(\cdot,t)\|_{L^{\infty}(\mathbb{T})} \leq CC_{0}^{\frac{1}{2}}, \label{eq:thm1-u2}\\
		&\sigma^{\frac{3}{4}+\frac{\theta}{2}}\left\langle\frac{\epsilon u_{x}}{v}-p\right\rangle_{\mathbb{T}\times[t,\infty)}^{\frac{1}{2},\frac{1}{4}} \leq CC_{0}^{\frac12},\label{eq:thm1-up}
	\end{align}
where $\sigma = \min\{1, t\}$ and $\theta \in (0,1)$. More comprehensive discussions concerning the H\"older regularity of solutions are delivered throughout the proof of local solutions in Section \ref{sec:local}.
\end{remark}

\begin{remark}
Compared with Hopf’s results \cite{hoff1991,hoff1992} on the existence of discontinuous solutions to the Cauchy problem for non-isothermal gases, our work imposes much weaker restrictions on the initial data. Specifically, we remove the smallness assumptions on the \(L^1\)-norm of initial perturbations and the magnitude of the jump discontinuity of the specific volume $v(\cdot, t)$, i.e., $\sum|[v_0(y_i)]|$. As a matter of fact, the magnitude of the jump discontinuities at $x=\pm l$ we investigate herein will not be small.
Taking advantage of the properties of periodic boundary conditions and the structural features of isothermal gases, we establish the global existence and long-time behavior of solutions to the perturbed equation under the initial condition \eqref{eq:vv-b0} and \eqref{eq:e0-b0}.
\end{remark}

\begin{remark}
We emphasize that even if the volume-averaged integral lies within the metastable regime, smooth solutions to system \eqref{Approximate ns-lagrange} are still non-unique. In particular, there exists at least one solution featuring two jump discontinuities; moreover, solutions with $2N$ jumps can emerge, where the integer $N$ is governed by the period length and the value of the volume-averaged integral (see \cite{chnys2026} for details). Solutions with $2N$ jump discontinuities represent multi-phase coexisting phase transition configurations. Importantly, we can also establish the asymptotic stability of these $2N$-jump solutions under small initial perturbations. Guided by the energy-minimizing principle, we restrict our attention to the well-posedness and long-time asymptotic behavior of two-smooth-interface phase transition solutions corresponding to the case $N=1$. 
\end{remark}
   
The rest of this paper is arranged as follows. In Section \ref{sec:local}, we will give the local existence of weak solution for   \eqref{ns-lagrange} with discontinuous initial data.  In Section \ref{sec3},  we will present the desired a priori estimates on the local solutions for Cauchy problem \eqref{ns-lagrange},  then give the proof of the main theorem.

\section{Local existence of discontinuous solution}\label{sec:local}
\indent\qquad
In this section, we will give the local existence of the discontinuous solution for the system  \eqref{ns-lagrange}.
Following the seminal method in \cite{hoff1991}, we construct a semi-discrete staggered grid difference scheme, derive uniform time-weighted energy estimates to overcome the initial singularity at \(t\to0^+\).
We state a local existence result.

\begin{theorem}%[Local Existence]
\label{thm:loc}
	Let $C_0, \underline{v}$ and $\hat{v}$ be as in Theorem \ref{thm1}, and  $M$ be a given positive number. Then there is a positive constant $C=C(\underline{v}, \hat{v}, M)$ and a time $t_0$, such that if the initial data \((v_0,u_0)\) satisfying
	\begin{align}
		& v_0\in [b+2\underline{v},\hat{v}-\underline{v}], ~~a.e., \label{eq:v-b0}\\
		& C_0 \leq M, \label{eq:c0-b0} %\\&[p(v_0)-\frac{\epsilon u_{0x}}{v_0}]=0,  \quad \text{at ~}x=\pm l, \label{eq:jump-b0}
	\end{align}
	%\begin{equation}
	%	[p(v)-\frac{\epsilon u_x}{v}]=0,  \quad \text{at each~}y_i,~~i=1,2,\dots,N_0.
	%\end{equation}
	the problem \eqref{ns-lagrange} has a local weak solution \((v,u)\) defined up to time $t_0$ satisfying
	\begin{align}
		& v\in [b+\underline{v},\hat{v}], ~~a.e., \label{eq:v-b1}\\
		& A(t_0)+B(t_0)+F(t_0)\leq C C_0, \label{eq:c0-b1}
	\end{align}
	and the solution has the following regularity:
	\begin{align}
		&\langle v\rangle_{I_i\times[0,t_{0}]}^{\frac{1}{2},\frac{1}{2}-\frac{\theta}{4}}\leq CC_{0}^{\frac{1}{2}},\quad i=1,2,3.\label{eq:loc-v1}\\		
		&\langle u\rangle_{\mathbb{T}\times[\tau,t_{0}]}^{\frac{1}{2},\frac{1}{4}} \leq CC_{0}^{\frac12}\tau^{-\frac{1}{4}-\frac{\theta}{2}},\label{eq:loc-u1} \\
		&\|u_{x}(\cdot,t)\|_{L^{\infty}(\mathbb{T})} \leq CC_{0}^{\frac{1}{2}}t^{-\frac{1}{2}-\frac{\theta}{4}}, \label{eq:loc-u2}\\
		&\left\langle\frac{\epsilon u_{x}}{v}-p\right\rangle_{\mathbb{T}\times[\tau,t_{0}]}^{\frac{1}{2},\frac{1}{4}} \leq CC_{0}^{\frac12}\tau^{-\frac{3}{4}-\frac{\theta}{2}}.\label{eq:loc-up} 
	\end{align}
In particular, \eqref{eq:loc-v1} and \eqref{eq:loc-up} show that \(v\) and \(u_{x}\) have one-sided limits at \(x=\pm l\), and that the jump condition \eqref{jump0} holds. Finally
	\begin{equation}\label{eq:jump-loc}
		[\ln v(\pm l,t)]=[\ln v(\pm l,0)]\exp\Big(\epsilon^{-1}\int_{0}^{t_0}\alpha_{\pm}(s)ds\Big),
	\end{equation}
	where \(\alpha_{\pm}\) is given by \eqref{eq:l3}.
\end{theorem}

\begin{remark}
We restrict our attention to the well-posedness of two-smooth-interface phase transition solutions, which contain exactly two jump discontinuities. In fact, we can also establish the local existence of solutions for cases with multiple discontinuities. Specifically, suppose the initial data $(v_0, u_0)$ possess finitely many jump discontinuities located at $x=y_i$, $i=1,2,\dots,N_0$. The desired result still holds if we merely replace the condition \eqref{eq:c10} with 
\begin{equation*}
\widetilde{\mathscr{C}}=\big(\mathop{\sum }\limits_{i}|[v_{y_i}](0)|\big)^{2}.
\end{equation*}
\end{remark}

\subsection{Difference approximations}
\indent\qquad
In this subsection, we establish the existence of certain semidiscrete difference approximations of solutions to the problem \eqref{ns-lagrange} and derive the a priori energy estimates required for extracting limiting solutions as the discretization tends to zero.

To begin, let \(N\in\mathbb{N}\) be a positive integer, and define the mesh size \(h=L/N\). The grid points are given by
\[
x_k=kh,~~ (k=0,\pm 1,\dots,\pm N), \quad x_j=jh,~~ (j=\pm\frac12,\pm\frac32,\dots, \pm(N-\frac12)).
\]
We use a staggered layout as \(u_k(t)\triangleq u(x_k,t)\) represents the velocity at cell vertices \(x_k\) and \(v_j(t)\triangleq v(x_j,t)\) represents the specific volume at cell centers \(x_{j}\).
By periodicity, \(u_k(t)=u_{2N+k}(t)\) and \(v_{j}(t)=v_{2N+j}(t)\). 

We define the difference operator \(\delta\) by
\[\delta w_l \triangleq \frac{w_{l+\frac12}-w_{l-\frac12}}h,\]
for \(l=k\) or \(j\), and let $N_*$ will be the integer nearest $l/h$, so that $x_{\pm N_* -1/2}\leq {\pm l}\leq x_{\pm N_* +1/2}$.
Approximations $(v_j,u_k)(t)$ are then computed from the ordinary differential equations
\begin{equation}\label{eq:scheme}\left\{
\begin{aligned} 
\displaystyle  &\dot{v_j}=\delta u_j,&&\quad j=\pm \frac12,\pm \frac32,\dots,\pm(N-\frac12),\\
&\dot{u_k}=-\delta p_k+\delta\left(\frac{\epsilon \delta u}{v} \right)_k, &&\quad k=0,\pm 1,\dots,\pm N,
\end{aligned}\right.
\end{equation}
where $p_j=p(v_j)$ and $\dot{w}\triangleq \dfrac{\mathrm{d}w}{\mathrm{d}t}$.  The periodic boundary conditions imply \(u_{-N}=u_N\) and \(v_{-N+1/2}=v_{N+1/2}\). 

We assume that initial values $(v_j(0),u_k(0))$ have been given and let $(\bar{v},\bar{u})$ be as in \eqref{Hypo-1}. In analogy with \eqref{eq:C0}, we define
\begin{equation}\label{eq:c00}
\begin{aligned}
  \mathscr{C}_{0}\triangleq &\sum_{k} (u_{k}(0)-\bar u)^{2}h+\sum_{j}(v_{j}(0)-\bar v)^{2}h+\Big(\sum_{j}|\delta u_{j}(0)|h\Big)^{2}%\\  &
  +\sum_{k^{\prime}}\delta v_{k}(0)^{2}h,
\end{aligned}
\end{equation}
and
\[\sum_{k}w_{k}\triangleq\sum_{k=-N}^{N-1}w_{k},\quad\sum_{k'}w_{k}\triangleq\sum_{k=-N,~k\neq \pm N_*   }^{N-1}w_{k},\quad\sum_{j}w_{j}\triangleq\sum_{j-\frac12=-N}^{N-1}w_{j}%,\quad\sum_{j}w_{j}\triangleq\sum_{i=1}^{N_0}[w]_{k_i}
.\]
We also define
\begin{equation}\label{eq:c10}
%{\mathscr{C}}_{1} = \mathop{\sum }\limits_{j}\delta {u}_{j}{\left( 0\right) }^{2}h,\quad\quad 
\widetilde{\mathscr{C}}=\big(|[v_{N_*}](0)|+|[v_{-N_*}](0)|\big)^{2},
\end{equation}
where \([w_{\pm N_* }]\triangleq w_{\pm N_* +\frac12}-w_{\pm N_* -\frac12}\).

It will be useful to adopt the following quantities, which are the discrete analogs of the norms \(A\), \(B\), and \(F\) defined in \eqref{eq:A}-\eqref{eq:F}:
\begin{align*}
\mathscr{A}(t_0) &= \sup_{0\leq t\leq t_0}\Big(\sum_{j} (v_j-\bar v)(t)^{2}h+\sum_{k}(u_k-\bar u)(t)^2h+ \sum_{k'}\delta v_{k}(t)^2h%\\&\quad\quad\quad\quad+\big(|[v]_{k_{+}}(0)+|[v]_{k_{-}}(0)|\big)^2
\Big) +\int_{0}^{t_{0}} \sum\limits_{j}\delta u_{j}(t)^{2}h dt,\\
\mathscr{B}(t_0) &= \sup_{0\leq t\leq t_{0}} t^{\frac{1}{2}}\sum_{j}\delta u_{j}(t)^{2}h +\int_{0}^{t_{0}}t^{\frac{1}{2}+\theta}\Big(\sum_{k}\dot{u}_{k}(t)^{2}h+\sum\limits_{k^{\prime}}\delta\left(\frac{\delta u}{v}\right)_{k}(t)^{2}h\Big)dt, \\
%\mathscr{D}(t_0) &= \sup_{0\leq t\leq t_{0}}  \sum_{j}\delta u_{j}(t)^{2}h +\int_{0}^{t_{0}} \Big(\sum_{k}\dot{u}_{k}(t)^{2}h+\sum\limits_{k^{\prime}}\delta\left(\frac{\delta u}{v}\right)_{k}(t)^{2}h\Big)dt, \\
\mathscr{F}(t_{0})& =\sup_{0\leq t\leq t_{0}} t^{\frac{3}{2}+\theta}\Big(\sum_{k}\dot{u}_{k}(t)^{2}h+\sum_{k^{\prime}}\delta\left(\frac{\delta u}{v}\right)_{k}(t)^{2}h\Big)+\int_{0}^{t_{0}} t^{\frac{3}{2}+\theta}\sum\limits_{j}\delta\dot{u}_{j}(t)^{2}h dt.
\end{align*}
Recall that $\theta$ is an arbitrarily small positive constant satisfying \(0<\theta<1\).
We also adopt the following notation to describe semidiscrete H\"{o}lder continuity: given a sequence \(\{w_{k}(t)\}_{k\in\mathbb{Z}}\) and a set \(S\) in \((x,t)\)-space, define
\[\langle w_{k}\rangle_{S}^{a,b}=\sup_{(x_{k},t),(x_{k^{\prime}},t^{\prime})\in S}\frac{|w(x_{k},t)-w(x_{k^{\prime}},t^{\prime})|}{|x_{k}-x_{k^{\prime}}|^{a}+|t-t^{\prime}|^{b}}.\]
The same notation applies in an obvious way to sequences \(\{w_{j}(t)\}_{j+1/2\in\mathbb{Z}}\).

The local existence of solutions of the scheme \eqref{eq:scheme} may then be formulated as follows: 
\begin{lemma}\label{thm-exsit2}
Let \(\theta, \underline{v}, \hat{v}\) and \(M_{0}\) be given positive numbers. Then there are constants \(C_1, C_2\), \(t_{0}\), and \(h_0\), depending only on $\theta, \underline{v},\hat v$ and $M_{0}$, such that: given \(h\leq h_{0}\) and initial data \(v_{j}(0), u_{k}(0)\) with values \(b+2\underline{v}\leq v_{j}( 0)\leq \hat v-\underline{v}\)
and with \(\mathscr{C} _{0}\leq M_{0}\), the scheme \eqref{eq:scheme} is solvable up to time \(t_0\) and satisfies
\begin{align}
\displaystyle  &b+\underline{v}\leq v_{j}\leq \hat v, \label{local-01} \\
&\mathscr{A}(t_0)+\mathscr{B}(t_0)\leq C_1\mathscr{C}_0, \label{local-02} %\\&\mathscr{A}(t_0)+\mathscr{D}(t_0)\leq C_2(\mathscr{C}_0+\mathscr{C}_1),\label{local-03}
\end{align}
and the regularity conditions
\begin{align}
&\langle v_{j}\rangle_{I_i\times[0,t_{0}]}^{\frac{1}{2},\frac{1}{2}-\frac{\theta}{4}}\leq C\mathscr{C}_{0},\quad i=1,2,3, \label{local-04}\\
%&\langle u_{k}\rangle_{\{0\leq t\leq t_{0}\}}^{\frac{1}{2},\frac{1}{4}}\leq C(\mathscr{C}_0+\mathscr{C}_1)^{1/2},\label{local-05}\\
&\langle u_k\rangle_{\{\tau\leq t\leq t_0\}}^{\frac{1}{2},\frac{1}{4}}\leq C\mathscr{C}_0\tau^{-\frac{1}{4}-\frac{\theta}{2}},\label{local-06}\\
&\left|\left[\ln v_{\pm N_* }(t)\right]-\exp\left(\epsilon^{-1}\int_{0}^{t}\alpha_{\pm}(s)ds\right)\left[\ln v_{\pm N_* }(0)\right]\right|\leq Ch^{1/2}\mathscr{C}_0^{1/2},\label{local-07}
\end{align}
where \(\alpha_{\pm}(s)\) is given by \eqref{eq:l3}.
\end{lemma}

\begin{proof}
Observe that if \(\mathscr{C}_{0}<\infty\) and $v_{j}(0)\in [b+2\underline{v},\hat{v}-\underline{v}] $, then \(v_j(0)\) and \(u_k(0)\) are bounded uniformly in \(j\) and \(k\) (with bounds depending on \(h\) ). Thus the ordinary  differential equations \eqref{eq:scheme} together with these initial values constitute a well-posed initial value problem in \(L^{\infty}\cap L^{2}\) and so have a local solution defined up to some \(t_0>0\) which may depend on $h$. Our goal is to show that \(t_{0}\) is in fact independent of \(h\). We therefore assume that $v_{j}(t)\in [b+\underline{v},\hat{v}]$ for all $j$ and all $t\in [0,t_0]$ for some positive $t_0$, and proceed to obtain bounds for $\mathscr{A}, \mathscr{B}$ and $\mathscr{D}$. 
Throughout this proof $C$ denotes a generic positive constant as described above. 

We first give the  $L^2$ estimate of the local solution to \eqref{eq:scheme}. We multiply \eqref{eq:scheme}$_{1,2}$ by \((v_j-\bar v)h\) and \((u_k-\bar u)h\), respectively, and sum to obtain
\begin{align*}
\displaystyle  &\frac12\frac{d}{dt}\Big(\sum_{j} (v_j-\bar v)^{2}h+\sum_{k}(u_k-\bar u)^2h\Big)\\&=\sum_{j}\delta u_j(v_j-\bar v)h+\sum_{k}\Big(-\delta(p-\bar p)_k(u_k-\bar u)h+\delta\big(\frac{\epsilon\delta u}{v}\big)_k(u_k-\bar u)h\Big)\\
&=\sum_{j}\delta u_j(v_j-\bar v)h+\sum_{j}\Big((p_j-\bar p)\delta u_jh-\frac{\epsilon\delta u_j^2}{v_j}h\Big),
\end{align*}
where $\bar{p}=p({\bar v})$. Integrating the above equation and using the assumption $b+\underline{v}\leq v_j\leq \hat{v}$, we have
\begin{equation}\label{eq:b1}
\begin{aligned}
\displaystyle  &\Big(\frac12\sum_{j} (v_j-\bar v)(t)^{2}h+\frac12\sum_{k}(u_k-\bar u)(t)^2h\Big)\Big|^t_0+\int_0^t \sum_{j}\delta u_j(s)^2hds\\
&\leq C\int_0^t\sum_{j}\Big(\delta u_j(v_j-\bar v)h+(p_j-\bar p)\delta u_jh\Big)ds.\\
\end{aligned}  
\end{equation}
Noting that \(p_j-\bar p=p'(v_*)(v_j-\bar v)\) and \eqref{eq:vvv}, we can bound the right side of \eqref{eq:b1} by
\[\begin{aligned}
&C\int_0^t\sum_{j}\Big(\delta u_j(v_j-\bar v)h+(p_j-\bar p)\delta u_jh\Big)ds\\& \leq C\int_{0}^{t}\sum_{j}|v_{j}-\bar v||\delta u_j|hds \\&\leq C\int_{0}^{t}\sum_{j}(v_{j}-\bar v)(s)^2hds+\frac12\int_{0}^{t}\sum_{j}\delta u_j(s)^2hds\\&\leq C\mathscr{A}t+\frac12\int_{0}^{t}\sum_{j}\delta u_j(s)^2hds.
\end{aligned}\]
Substituting the above formula into \eqref{eq:b1} yields
\begin{equation}\label{eq:b0}
  \sum_{j}(v_{j}-\bar v)(t)^{2}h+\sum_{k}(u_{k}-\bar u)(t)^{2}h+\int_{0}^{t}\sum\limits_{j}\delta u_{j}(s)^{2}hds\leq C({\mathscr{C}}_{0}+{\mathscr{A}}t).
\end{equation} 

Next, we derive an upper bound for \(\sum\limits_{k^{\prime}}\delta v_{k}^{2}h\).
Letting \(L_j=L(v_j)=\ln v_j\), we find from \eqref{eq:scheme}$_1$ that
\begin{equation}\label{eq:vx1}
  \dot{L}_{j}=\frac{\dot v_{j}}{v_{j}}=\frac{\delta u_{j}}{v_{j}},
\end{equation}
thus we have from \eqref{eq:scheme}$_2$ and \eqref{eq:vx1} that
\begin{equation}\label{eq:vx2}
  \epsilon\delta\dot{L}_{k}=\delta\Big(\frac{\epsilon\delta u}{v}\Big)_{k}=\delta p_k+\dot{u}_k.
\end{equation}
It is easy to obtain from the assumption $b+\underline{v}\leq v_j\leq \hat{v}$ and \eqref{eq:vvv} that
\begin{equation}\label{eq:lpv}
	\delta L_k\sim\delta p_{k}\sim\delta v_{k}.
\end{equation}
Multiplying \eqref{eq:vx2} by \(\delta L_{k}\), summing and integrating, we obtain
\begin{equation*}
\begin{aligned}\frac\epsilon2\sum_{k^{\prime}}\delta L_{k}^{2}h\Big|_{0}^{t}&=\int_{0}^{t}\sum_{k^{'}}\delta L_{k}\delta p_{k}hds+\sum_{k^{'}}(u_k-\bar u)\delta L_{k}h\Big|_{0}^{t}%\\&
-\frac1\epsilon\int_{0}^{t}\sum_{k^{'}} (u_k-\bar u)(\dot{u}_{k}+\delta p_{k})hds\\&
=-\frac1{2\epsilon}\sum_{k^{\prime}} (u_k-\bar u)^{2}h\Big|_{0}^{t}+\int_{0}^{t}\sum_{k^{'}}\delta L_{k}\delta p_{k}hds+\sum_{k^{'}}(u_k-\bar u)\delta L_{k}h\Big|_{0}^{t}\\&\quad
-\frac1\epsilon\int_{0}^{t}\sum_{k^{'}} (u_k-\bar u)\delta p_{k}hds\\
&\leq C\mathscr{C}_{0}+C\sum_{k^{\prime}} (u_k-\bar u)^{2}h+\frac\epsilon4\sum_{k^{\prime}}\!\delta L_{k}^{2}h+\int_{0}^{t}\!\sum_{k^{'}}\!\delta v_{k}^{2}hds+\int_{0}^{t}\!\sum_{k^{'}}(u_{k}\!-\!\bar u)^{2}hds\\
&\leq C\mathscr{C}_{0}+C\sum_{k^{\prime}} (u_k-\bar u)^{2}h+\frac\epsilon4\sum_{k^{\prime}}\!\delta L_{k}^{2}h+\mathscr{A}t,\end{aligned}  
\end{equation*}
which follows from the definition of $\mathscr{A}$, \eqref{eq:vx2} and \eqref{eq:lpv}. Elementary estimates based upon \eqref{eq:b0} then show that
\begin{equation}\label{eq:vx0}
\begin{aligned}\sum_{k'}\delta v_{k}(t)^{2}h\leq C(\mathscr{C}_{0}+\mathscr{A}t).\end{aligned}  
\end{equation}
We obtain from \eqref{eq:b0} and \eqref{eq:vx0} that
 \begin{equation}\label{eq:est-a0}
  \mathscr{A}(t)\leq C(\mathscr{C}_0+\mathscr{A}t).
\end{equation}
Moreover, we find from \eqref{eq:vx2} that at \(x_{\pm N_* }\),
\begin{equation}\label{eq:l1}
  \epsilon[\dot{L}_{\pm N_* }]=\Big[\big(\frac{\epsilon\delta u}{v}\big)_{\pm N_* }\Big]=[p_{\pm N_* }]+h\dot{u}_{\pm N_* },
\end{equation}
where \([w_{\pm N_* }]\triangleq w_{\pm N_* +\frac12}-w_{\pm N_* -\frac12}\). We obtain
\begin{equation}\label{eq:l2}
 [p_{\pm N_* }]=\alpha_\pm[L_{\pm N_* }],
\end{equation}
in which 
\begin{equation}\label{eq:l3}
 \alpha_{\pm}=\frac{p(v_{\pm N_* +\frac12})-p(v_{\pm N_* -\frac12})}{L(v_{\pm N_* +\frac12})-L(v_{\pm N_* -\frac12})}=\xi_\pm p'(\xi_\pm),
\end{equation}
with $\xi_\pm$ are between $v_{\pm N_* -\frac12}$ and $v_{\pm N_* +\frac12}$. 
The bounds in \eqref{eq:vvv} show that $-C^{-1}\leq\alpha_\pm\leq C$. Equation \eqref{eq:l1} thus becomes
\begin{equation}\label{eq:l4}
  \epsilon[\dot{L}_{\pm N_* }]=\alpha_\pm[L_{\pm N_* }]+h\dot{u}_{\pm N_* },
\end{equation}
whose solution is
\[\begin{aligned}
[L_{\pm N_* }(t)]& =\exp\Big(\epsilon^{-1}\int_{0}^{t}\alpha_{\pm}(s)ds\Big)\Big([L_{\pm N_* }(0)]-\frac{h}{\epsilon}(u_{\pm N_* }(0)-\bar u)\Big)+\frac{h}{\epsilon}(u_{\pm N_* }(t) -\bar u) \\
&+\frac{h}{\epsilon^{2}}\int_{0}^{t}\alpha_{\pm}(s)\exp\Big(\epsilon^{-1}\int_{s}^{t}\alpha_{\pm}(\tau)d\tau\Big)(u_{\pm N_* }(s) -\bar u)ds.
\end{aligned}\]
The bound \(|u_{\pm N_* }-\bar u|\leq h^{-1/2}\mathscr{A}^{1/2}\) and the assumption $0\leq t\leq t_0\leq 1$ show that
\begin{equation}\label{eq:lnv-jump0}\Big|[L_{\pm N_* }(t)]-\exp\Big(e^{-1}\int_{0}^{t}\alpha_{i}(s)ds\Big)[L_{\pm N_* }(0)]\Big|\leq Ch^{\frac12}\mathscr{A}^{\frac12},\end{equation}
so that by the above bound for \(\alpha_\pm\), it holds
\begin{equation}\label{eq:v-jump0}\left|[v_{\pm N_* }(t)]\right|\leq C(\left|[v_{\pm N_* }(0)]\right|+h^{\frac12}\mathscr{A}^{\frac12}).\end{equation}
We also have the estimate
\begin{equation}\label{eq220}
|[p_{\pm N_* }(t)]|\leq C(\tilde{\mathscr{C}}^{1/2}+h^{\frac12}\mathscr{A}^{\frac12}),
\end{equation}
which follows from \eqref{eq:l2} and \eqref{eq:v-jump0}.

We now proceed to derive a bound for $\mathscr{B}$. 
We multiply the second equation in \eqref{eq:scheme} by \(\delta\psi_{k}(s)h\)  where \(\{\psi_{j}(s)\}_{0\leq s\leq t}\) is a test sequence to be specified below, sum the product over \(j\) and integrate theresult to obtain
\[-\sum_{j}\psi_{j}\delta u_{j}h\Big|_{0}^t=-\int_{0}^{t}\sum_{j}\delta u_{j}\left(\dot{\psi}_{j}+\frac{\epsilon\delta(\delta\psi)_{j}}{v_{j}}\right)hds-\int_{0}^{t}\sum_{k}\delta\psi_{k}\delta p_{k}hds.\]
We now choose \(\psi_{j}\) to satisfy the initial value problem
\begin{equation}\label{back-para}
  \left\{\begin{aligned}
  \displaystyle  &\dot{\psi}_{j}+\frac{\epsilon\delta(\delta\psi)_{j}}{v_{j}}=0,\quad0\leq s\leq t,\\&\psi_{j}(t)=\delta u_{j}(t)\Big/\Big(\sum_{j}\delta u_{j}(t)^{2}h\Big)^{1/2}.
  \end{aligned}  \right.
\end{equation}
We then have that
\begin{equation}\label{eq228}
 \begin{aligned}
\Big(\sum_{j}\delta u_{j}(t)^{2}h\Big)^{1/2}\leq & \Big|\sum_{j}\psi_{j}(0)\delta u_{j}(0)h\Big|+\int_{0}^{t}\sum_{k^{\prime}}|\delta\varphi_{k}\delta p_{k}|hds \\&+\int_{0}^{t}\|\delta\psi(s)\|_{\infty}(|[p_{N_* }(s)]|+|[p_{- N_* }(s)]|)ds\\\triangleq &\sum_{i=1}^3I_i.
\end{aligned} 
\end{equation}
Defining the quantity $\mathscr{E}$ by
\[\begin{aligned}\label{eq229}
\mathscr{E}=& \sup_{0\leq s\leq t}\Big(\sum_{j}\psi_{j}(s)^{2}h+(t-s)\sum_{k}\delta\psi_{k}(s)^{2}h\Big)  \\
&+\int_{0}^{t}\Big(\sum\limits_{k}\delta\psi_{k}(s)^{2}h+(t-s)\sum_{j}(\dot{\psi}_{j}(s)^{2}+\delta(\delta\psi)_{j}(s)^{2})h\Big)ds.
\end{aligned}\]
We proceed to estimate the right-hand side of \eqref{eq228} in terms of $\mathscr{E}$. First,
\[\|\psi(0)\|_{\infty}\leq \Big(\sum_{j}\psi_{j}(0)^{2}h\Big)^{1/4}\Big(\sum_{k}\delta\psi_{k}(0)^{2}h\Big)^{1/4}\leq Ct^{-1/4}\mathscr{E}^{1/2},\]
so that
\[\begin{aligned}
I_1=\Big|\sum\psi_{j}(0)\delta u_{j}(0)h \Big|&\leq Ct^{-1/4}\mathscr{E}^{1/2}\sum_{j}\mid u_{j+\frac{1}{2}}(0)-u_{j-\frac{1}{2}}(0)\mid \\
&\leq Ct^{-1/4}\mathscr{E}^{1/2}\mathscr{C}_{0}^{1/2}. 
\end{aligned}\]
Next, from \eqref{eq:lpv} and \eqref{eq229},
\[\begin{aligned}
I_2=\int_{0}^{t}\sum_{k^{\prime}}|\delta\varphi_{k}|~|\delta p_{k}| hds &\leq C\Big(\int_{0}^{t}\sum\limits_{k^{\prime}}\delta v_{k}^{2}hds\Big)^{1/2}\mathscr{E}^{1/2} \leq C\mathscr{A}^{1/2}\mathscr{E}^{1/2}.
\end{aligned}\]
Finally, using \eqref{eq229} and \eqref{eq220}, we have
\[\begin{aligned}
I_3&\leq\int_{0}^{t}\Big(\sum_{k}\delta\psi_{k}^{2}h\Big)^{1/4}\Big(\sum_{j}\delta(\delta\psi)_{j}^{2}h\Big)^{1/4}(\tilde{\mathscr{C}}^{1/2}+h^{1/2}{\mathscr{A}}^{1/2})ds \\
&\leq C\mathscr{E}^{\frac14}\Big(\int_{0}^{t}(t-s)^{-\frac23}ds\Big)^{\frac34}\Big(\int_{0}^{t}\sum_{j}(t-s)\delta(\delta\psi)_{j}^{2}hds\Big)^{\frac14}(\tilde{\mathscr{C}}^{\frac12}+h^{\frac12}{\mathscr{A}}^{\frac12})\\
&\leq C\mathscr{E}^{1/2}t^{\frac14}(\tilde{\mathscr{C}}^{\frac12}+h^{\frac12}{\mathscr{A}}^{\frac12}).
\end{aligned}\]
Substituting the above estimates into \eqref{eq228} yields
\begin{equation}\label{eq230}
\begin{aligned}
  \sup\limits_{0\leq t\leq t_{0}}t^{1/2}\sum\limits_{j}\delta u_{j}(t)^{2}h&\leq C{\mathscr{E}}(\mathscr{C}_{0}+\mathscr{A}t^{1/2}+\tilde{\mathscr{C}}t+ h\mathscr{A}t)\\
  &\leq C{\mathscr{E}}(\mathscr{C}_{0}+\mathscr{A}t_0^{1/2}+ h\mathscr{A}),
\end{aligned}
\end{equation}
by the assumption $\tilde{\mathscr{C}}t\leq \mathscr{C}_{0}$ for small time $t\leq t_0$.

We now derive  the estimate of $\mathscr{E}$ in terms of $\mathscr{A}$.
First, we multiply \eqref{back-para} by \(v_j\psi_jh\), sum and integrate to obtain
$$\sum_{j}\psi_{j}(s)^{2}h+\int_{s}^{t}\sum\limits_{k}\delta\psi_{k}(\tau)^{2}hd\tau\leq C\Big(1+\int_{s}^{t}\sum_{j}\psi_{j}^{2}|\delta u_{j}|hd\tau\Big).$$
The last term can be bounded by
\begin{equation*}
\begin{aligned}
\int_{s}^{t}\sum_{j}\psi_{j}^{2}|\delta u_{j}|hd\tau
&\leq\int_{s}^{t}\Big(\sum_{j}\psi_{j}(\tau)^{2}h\Big)^{1/2}\Big(\sum_{j}\delta u_{j}(\tau)^{2}h\Big)^{1/2}\|\psi_j\|_\infty d\tau\\
&\leq C{\mathscr{E}}\int_{s}^{t}\Big(\sum_{j}\delta u_{j}(\tau)^{2}h\Big)^{1/2}(t-\tau)^{-\frac14} d\tau\\
&\leq C{\mathscr{E}}{\mathscr{A}}^{1/2}t^{1/4},
\end{aligned}
\end{equation*}
thus
\begin{equation}\label{eq233}
\begin{aligned}
\sum_{j}\psi_{j}(s)^{2}h+\int_{s}^{t}\sum\limits_{k}\delta\psi_{k}(\tau)^{2}hd\tau
&\leq C(1+{\mathscr{E}}{\mathscr{A}}^{1/2}t^{1/4}).
\end{aligned}
\end{equation}
Multiplying \eqref{back-para} by \(\left(t-s\right)v_{j}\dot{\psi}_{j}h\), summing
and integrating, we obtain in a similar way that
\begin{equation}\label{eq234}
(t-s)\sum\limits_{k}\delta\psi_{k}(s)^{2}h+\int_{s}^{t}\sum\limits_{j}(t-\tau)\dot{\psi}_{j}(\tau)^{2}hd\tau\leq C\int_{s}^{t}\sum\limits_{k}\delta\psi_{k}(\tau)^{2}hd\tau.\end{equation}
Combining appropriate multiples of \eqref{eq233}, \eqref{eq234} and \eqref{back-para}, and taking the
supremum over \(s\in[0,t]\), we obtain
\begin{equation}\label{eq231}
\mathscr{E}\leq C(1+\mathscr{E}\mathscr{A}^{1/2}t^{1/4}).  
\end{equation}
It holds that \(\mathscr{A}t^{1/2}\leq1/(4C^{2})\) for small time $t\leq t_0$, which shows \(\mathscr{E}\leq C\). Substituting this into \eqref{eq230}, we obtain 
\begin{equation}\label{eq:est-ux}
\sup_{0\leq t\leq t_{0}}t^{1/2}\sum_{j}\delta u_{j}(t)^{2}h\leq C(\mathscr{C}_{0}+\mathscr{A}t_0^{1/2}+ h\mathscr{A}).\end{equation}

Next, we multiply the second equation in \eqref{eq:scheme} by \(t^{\beta}\dot{u}_{k}(t)h\), sum and integrate to get
\begin{equation}\label{eq:t-beta}\begin{aligned}
		&\frac{s^{\beta}}{2}\sum_{j}\frac{\delta u_{j}(s)^{2}h}{v_j}\Big|_0^t+\int_{0}^{t}\sum_{k}s^{\beta}\dot{u}_{k}(s)^{2}hds \\&= s^{\beta}\sum_{j}(p_j-\bar p)\delta u_j h\Big|_0^t+\int_0^t\frac{\beta s^{\beta-1}}{2}\sum_{j}\frac{\delta u_{j}(s)^{2}h}{v_j}ds%&\\
		-\int_0^t\frac{s^{\beta}}{2}\sum_{j}\frac{\delta u_{j}(s)^{3}h}{v_j^2}ds \\&
		\quad-\int_{0}^{t}s^{\beta}\sum_{j}\dot{p_j}\delta u_j hds-\int_{0}^{t}\beta s^{\beta-1}\sum_{j}(p_j-\bar p)\delta u_j hds
\end{aligned}\end{equation}
Let $\beta=\frac{1}{2}+\theta$ with \(0<\theta<1\), and after some elementary manipulations, we obtain
\begin{equation}\label{eq235}\begin{aligned}
&t^{\frac{1}{2}+\theta}\sum_{j}\delta u_{j}(t)^{2}h+\int_{0}^{t}\sum_{k}s^{\frac{1}{2}+\theta}\dot{u}_{k}(s)^{2}hds \\
&\leq t^{\frac{1}{2}+\theta}\sum_{j}|p_j-\bar p|\,|\delta u_j| h+C\int_0^ts^{-\frac{1}{2}+\theta}\sum_{j}\delta u_{j}(s)^{2}hds\\
&\quad		+C\int_0^ts^{\frac{1}{2}+\theta}\sum_{j}|\delta u_{j}(s)|^{3}hds +\int_{0}^{t}s^{\frac{1}{2}+\theta}\sum_{j}|\dot{p_j}|\,|\delta u_j| hds\\&\quad+C\int_{0}^{t}s^{-\frac{1}{2}+\theta}\sum_{j}|p_j-\bar p|\,|\delta u_j| h \\&\triangleq \sum_{i=1}^{5}J_i.
\end{aligned}\end{equation}
We apply straightforward estimates to each term on the right-hand side of \eqref{eq235}. First, By \eqref{eq:vvv},\eqref{eq:lpv} and the definition of $\mathscr{A},\mathscr{B}$, we obtain
\begin{align}
	&J_1\leq C t^{\frac{1}{2}+\theta}\sum_{j}|v_j-\bar v|\,|\delta u_j| h\leq C\mathscr{A}^{1/2}\mathscr{B}^{1/2}t^{\frac{1}{4}+\theta},\label{eq:ux-j1}\\
	&J_2\leq C\mathscr{B}\int_0^ts^{-1+\theta}ds\leq C\mathscr{B}t^\theta,\\
	&J_4\leq C\int_{0}^{t}s^{\frac{1}{2}+\theta}\sum_{j}|\dot{v_j}|\,|\delta u_j| hds\leq C\mathscr{B}\int_0^ts^{\theta}ds\leq C\mathscr{B}t^{1+\theta},\\
	&J_5\leq C\int_0^ts^{-\frac{1}{2}+\theta}\sum_{j}|v_j-\bar v|\,|\delta u_j| hds%\leq C\mathscr{A}^{1/2}\mathscr{B}^{1/2}\int_0^ts^{-\frac{3}{4}+\theta}ds
	\leq C\mathscr{A}^{1/2}\mathscr{B}^{1/2}t^{\frac{1}{4}+\theta}.\label{eq:ux-j5}
\end{align}
Similarly, we have
\begin{equation}\label{eq:ux-j3}
\begin{aligned}
J_3&\leq s^{\frac{1}{2}+\theta}\int\sum |\delta u_{j}|^{3}hds  \\
&\leq C\int_{0}^{t}s^{\theta}\sum_{j}\delta u_{j}^{2}hds+C\int_{0}^{t}s^{1+\theta}\sum_{j}\delta u_{j}^{4}hds \\
&\leq C\mathscr{B}t^{\frac{1}{2}+\theta}+C\mathscr{B}\int_{0}^{t}s^{\frac12+\theta}\|\delta u\|_{\infty}^2(s)ds.
\end{aligned}	
\end{equation}
Hence we need to obtain the bound for \(\|\delta u\|_{\infty}\). A discrete Sobolev-type inequality shows that
\[\begin{aligned}
\left({\frac{\delta u}{v}}\right)_J^{2}-\left({\frac{\delta u}{v}}\right)_j^{2}& =\sum_{j<k<J}\Big(\big(\frac{\delta u}{v}\big)_{k+\frac{1}{2}}^{2}-\big(\frac{\delta u}{v}\big)_{k-\frac{1}{2}}^{2}\Big) \\
&\leq C\Big(\sum_{j}\delta u_{j}^{2}h\Big)^{1/2}\Big(\sum_{k^{\prime}}\delta\big(\frac{\delta u}{v}\big)_{k}^{2}h\Big)^{1/2} \\
&\quad+C\Big(\eta\left\|\delta u\right\|_{\infty}^{2}+\eta^{-1}\big(\big|\big[\frac{\delta u}{v}\big]_{\pm N_* }\big|\big)^{2}\Big).
\end{aligned}\]
Summing the above inequality over $j$ from $-N+1/2$ to $N-1/2$, we obtain
\[\begin{aligned}
\left({\frac{\delta u}{v}}\right)_J^{2}&\leq C\sum_{j}\delta u_{j}^{2}h+
 C\Big(\sum_{j}\delta u_{j}^{2}h\Big)^{1/2}\Big(\sum_{k^{\prime}}\delta\big(\frac{\delta u}{v}\big)_{k}^{2}h\Big)^{1/2} \\
&\quad+C\Big(\eta\left\|\delta u\right\|_{\infty}^{2}+\eta^{-1}\big(\big|\big[\frac{\delta u}{v}\big]_{N_* }\big|+\big|\big[\frac{\delta u}{v}\big]_{-N_* }\big|\big)^{2}\Big).
\end{aligned}\]
Taking $\eta$ small and applying \eqref{eq:l1} and \eqref{eq220}, we thus obtain
\begin{equation}\label{eq:ux-infty}
\begin{aligned}
\|\delta u(t)\|_{\infty}^{2}\leq Ct^{-\frac12}\mathscr{B}+Ct^{-\frac14}\mathscr{B}^{1/2}\Big(\sum_{k^{\prime}}\delta\big(\frac{\delta u}{v}\big)_{k}^{2}h\Big)^{1/2}+C\big(\tilde{\mathscr{C}}+ h\mathscr{A}+h^{2}(\dot{u}_{N_* }^{2}+\dot{u}_{-N_* }^{2})\big).
\end{aligned}	
\end{equation}
% Also, it holds
% \begin{equation}\label{eq:ux-infty1}
% \begin{aligned}
% \|\delta u(t)\|_{\infty}^{2}\leq C\mathscr{D}+C\mathscr{D}^{1/2}\Big(\sum_{k^{\prime}}\delta\big(\frac{\delta u}{v}\big)_{k}^{2}h\Big)^{1/2}+C\big(\tilde{\mathscr{C}}+ h\mathscr{A}+h^{2}(\dot{u}_{N_* }^{2}+\dot{u}_{-N_* }^{2})\big).
% \end{aligned}	
% \end{equation}
From above estimates, we have
\begin{equation}\label{eq:ux-j31}
\begin{aligned}
&\int_{0}^{t}s^{\frac12+\theta}\|\delta u\|_{\infty}^2(s)ds\\
&\leq C\int_{0}^{t}s^{\theta}\mathscr{B}ds+C\int_{0}^{t}s^{\frac14+\theta}\mathscr{B}^{1/2}\Big(\sum_{k^{\prime}}\delta\big(\frac{\delta u}{v}\big)_{k}^{2}h\Big)^{1/2}ds\\
&\quad+C\int_{0}^{t}s^{\frac12+\theta}\big(\tilde{\mathscr{C}}+ h\mathscr{A}\big)ds+C\int_{0}^{t}s^{\frac12+\theta}h^{2}(\dot{u}_{N_* }^{2}+\dot{u}_{-N_* }^{2})ds\\
&\leq C\mathscr{B}t^{1+\theta}+C\mathscr{B}t^{\frac12+\frac\theta 2}+C\big(\tilde{\mathscr{C}}+ h\mathscr{A}\big)t^{\frac32+\theta}+C\mathscr{B}h.
\end{aligned}	
\end{equation}
Substituting \eqref{eq:ux-j31} into \eqref{eq:ux-j3} yields
\begin{equation}\label{eq:ux-j32}
\begin{aligned}
J_3&\leq  C\mathscr{B}t^{\frac{1}{2}+\theta}+C\mathscr{B}\Big(\mathscr{B}t^{1+\theta}+\mathscr{B}t^{\frac12+\frac\theta 2}+\big(\tilde{\mathscr{C}}+ h\mathscr{A}\big)t^{\frac32+\theta}+\mathscr{B}h\Big)\\
&\leq  C(\mathscr{B}+\mathscr{B}^2)t^{\theta}+C h(\mathscr{A}^2+\mathscr{B}^2).
\end{aligned}	
\end{equation}
By combining \eqref{eq:ux-j1}-\eqref{eq:ux-j5}, \eqref{eq:ux-j32} and \eqref{eq235}, we obtain
\begin{equation}\label{eq236}\begin{aligned}
&t^{\frac{1}{2}+\theta}\sum_{j}\delta u_{j}(t)^{2}h+\int_{0}^{t}\sum_{k}s^{\frac{1}{2}+\theta}\dot{u}_{k}(s)^{2}hds \\
&\leq C\mathscr{A}^{1/2}\mathscr{B}^{1/2}t^{\frac{1}{4}+\theta}+C\mathscr{B}t^{1+\theta}+C(\mathscr{B}+\mathscr{B}^2)t^{\theta}+C h(\mathscr{A}^2+\mathscr{B}^2)\\
&\leq C(\mathscr{A}+\mathscr{B}+\mathscr{B}^2)t^{\theta}+C h(\mathscr{A}^2+\mathscr{B}^2),
\end{aligned}\end{equation}
which yields
\begin{equation}\label{eq:est-uxadd}\begin{aligned}
&\sup_{0\leq t\leq t_{0}}t^{\frac{1}{2}+\theta}\sum_{j}\delta u_{j}(t)^{2}h+\int_{0}^{t_0}\sum_{k}t^{\frac{1}{2}+\theta}\dot{u}_{k}(t)^{2}hdt\\
&\quad\leq C(\mathscr{A}+\mathscr{B}+\mathscr{B}^2)t_0^{\theta}+C h(\mathscr{A}^2+\mathscr{B}^2).
\end{aligned}\end{equation}
By virtue of the relations \eqref{eq:scheme},  \eqref{eq:vx0} and \eqref{eq:est-uxadd}, we obtain
\begin{equation*}
	\begin{aligned}
\int_{0}^{t_{0}}t^{\frac12+\theta}\sum\limits_{k^{\prime}}\delta\left(\frac{\delta u}{v}\right)_{k}(t)^{2}hdt&\leq \int_{0}^{t_{0}}t^{\frac12+\theta}\sum_{k^{\prime}}\dot{u}_{k}(t)^{2}hdt+\int_{0}^{t_{0}}t^{\frac12+\theta}\sum_{k^{\prime}}\delta p_{k}(t)^{2}hdt \\
&\leq C\mathscr{C}_0+C(\mathscr{A}+\mathscr{B}+\mathscr{B}^2)t_0^{\theta}+C h(\mathscr{A}^2+\mathscr{B}^2),
\end{aligned}
\end{equation*}
which together with \eqref{eq:est-uxadd} yields 
\begin{equation}\label{prop-ab}\mathscr{B}(t)\leq C\big(\mathscr{C}_0+(\mathscr{A}+\mathscr{B}+\mathscr{B}^2)t^{\frac{\theta}{2}}+ h(\mathscr{A}+\mathscr{A}^2+\mathscr{B}^2)\big).\end{equation}

Let \(C=C(\theta,\underline{v},\hat v,M_0)\) be the constant in \eqref{eq:est-a0}  and \eqref{prop-ab}. For small time $t_0$, we also choose \(h_{0}\) sufficiently small such that \(h\leq h_{0}\) and \(t\leq t_0\), there exists a constant $C_1=C_1(\theta,\underline{v},\hat v,M_0)$ satisfying
\[\mathscr{A}+\mathscr{B}\leq C_1\mathscr{C}_0.\]
Furthermore, by \eqref{eq:ux-infty}, it holds
\begin{equation}\label{eq:v-infty}
	\begin{aligned}
\left| {{v}_{j}\left( {t}\right)  - {v}_{j}\left( 0\right) }\right|  &\leq  \int_{0}^{{t}}\parallel {\delta u}\left( s\right) {\parallel }_{\infty }{ds}\\
&\leq  C\int_{0}^{t}  {s}^{-\frac{1}{4}}{\mathscr{B}}^{\frac{1}{2}}ds+C\int_{0}^{t}  {{s}^{-\frac{1}{4} - \frac{\theta }{4}}{\mathscr{B}}^{\frac{1}{4}}{\left( {s}^{\frac{1}{2} + \theta }\mathop{\sum }\limits_{k^{\prime}}\delta {\left( \frac{\delta u}{v}\right) }_{k}^{2}h\right) }^{\frac{1}{4}}}ds\\
&\quad +C\int_{0}^{t}\big(\tilde{\mathscr{C}}+ h\mathscr{A}+h^{2}\dot{u}_{\pm N_* }(t)^{2}\big)^{\frac12}ds\\
&\leq  C\Big({\mathscr{B}}^{\frac{1}{2}}{t}^{\frac{3}{4}}+  {\mathscr{B}}^{\frac{1}{2}}{t}^{\frac{1}{2} - \frac{\theta }{4}} + (\tilde{\mathscr{C}}+ h\mathscr{A})^{\frac{1}{2}}t+ {\mathscr{B}}^{\frac{1}{2}}{t}^{\frac{3}{4} - \frac{\theta }{2}} \Big)\\
&\leq  C{\mathscr{C}}_{0}^{\frac{1}{2}}t^{\frac{1}{2} - \frac{\theta }{4}}\leq \underline{v},
	\end{aligned}	
\end{equation}
provided that \(t\leq t_0\) and \(h\leq h_{0}=O(t_0)\). Thus
\begin{equation}\label{eq:v-infty1}
	b+\underline{v}\leq v_j(0)-\underline{v}\leq v_j(t)\leq v_j(0)+\underline{v}\leq \hat{v}.
\end{equation}
The proof of \eqref{local-01}-\eqref{local-02} then proceeds as above.

Next, we derive the bounds \eqref{local-04}-\eqref{local-07} as consequences of the energy estimates \eqref{local-02}. This will show that \eqref{local-04}-\eqref{local-07} hold, provided that \(h\leq h_{0}\) and \(t\leq t_0\), and for as long as the solution has values \(v_{j} \in [b+\underline{v}, \hat{v}]\). 
First, the validity of the assertions \eqref{local-04}-\eqref{local-06} of H\"{o}lder continuity in the \(x\)-variable is evident from the energy estimates \eqref{local-02}, and the definitions \(\mathscr{A},\mathscr{B}\). To prove the regularity of \(\left\{{v}_{j}\right\}\) in \(t\), we compute from \eqref{eq:scheme} and \eqref{eq:ux-infty} that
\begin{equation}\label{eq256}
	\begin{aligned}
\left| {{v}_{j}\left( {t}^{\prime }\right)  - {v}_{j}(t) }\right|  &\leq  \int_{t}^{{t}^{\prime }}\parallel {\delta u}\left( s\right) {\parallel }_{\infty }{ds}
% \\
% &\leq C\int_{t}^{{t}^{\prime}}  {s}^{-\frac{1}{4}}{\mathscr{B}}^{\frac{1}{2}}ds+ C\int_{t}^{{t}^{\prime}}  {{s}^{-\frac{1}{4} - \frac{\theta }{4}}{\mathscr{B}}^{\frac{1}{4}}{\left( {s}^{\frac{1}{2} + \theta }\mathop{\sum }\limits_{k^{\prime}}\delta {\left( \frac{\delta u}{v}\right) }_{k}^{2}h\right) }^{\frac{1}{4}}}ds\\
% &\quad +C\int_{t}^{{t}^{\prime}}\big(\tilde{\mathscr{C}}+ h\mathscr{A}+h^{2}\dot{u}_{\pm N_* }(t)^{2}\big)^{\frac12}ds\\
% &\leq  C\Big( {\mathscr{B}}^{\frac{1}{2}}\big| {\left( {t}^{\prime }\right) }^{\frac{3}{4}} - {t}^{\frac{3}{4}}\big|+ {\mathscr{B}}^{\frac{1}{2}}{\big| {\left( {t}^{\prime }\right) }^{\frac{2 - \theta }{3}} - {t}^{\frac{2 - \theta }{3}}\big| }^{\frac{3}{4}} + (\tilde{\mathscr{C}}+ h\mathscr{A})^{\frac{1}{2}}\big| {{t}^{\prime } - t}\big|\\
% &\quad\quad+ (h\mathscr{D})^{\frac{1}{2}}\big| {({t}^{\prime}) }^{\frac{1}{2}} - {t}^{\frac{1}{2}}\big| \Big)\\&
\leq  C \mathscr{C}_{0}^{\frac{1}{2}}{\big| {t}^{\prime } - t\big| }^{\frac{1}{2} - \frac{\theta }{4}}.
	\end{aligned}	
\end{equation}
We have taken \(h\) and \(t\) small here, depending on \({M}_{0}\).
To prove the H\"{o}lder continuity \eqref{local-06} for $u$, we first observe that, for \(J \in  \mathbb{Z}\) and \(s \geq  \tau\),
\[
\left| {u_k(s)  - \frac{1}{Jh}\mathop{\sum }\limits_{{l = k}}^{{k + J - 1}}{u}_{k}\left( s\right) h}\right|  \leq  \mathop{\sum }\limits_{{j = k + 1/2}}^{{k + J - 3/2}}\left| {\delta {u}_{j}}\right| h \leq  \tau^{\frac14}\mathscr{B}^{1/2}(Jh)^{1/2}.
\]
Thus
\[\begin{aligned}
\displaystyle 
\left| {{u}_{k}\left( {t}^{\prime }\right)  - {u}_{k}(t) }\right| & \leq  \frac{1}{Jh}\int_{t}^{{t}^{\prime }}\mathop{\sum }\limits_{k}^{{k + J - 1}}\left| {{\dot{u}}_{l}\left( s\right) }\right| {hds} + 2\tau^{-\frac14}\mathscr{B}^{1/2}(Jh)^{1/2}\\
&\leq C\tau^{-\frac14-\frac{\theta}{2}}\mathscr{B}^{1/2}(Jh)^{-1/2}|t'-t|^{\frac12}+2\tau^{-\frac14}\mathscr{B}^{1/2}(Jh)^{1/2}\\
& \leq  C\tau^{-\frac14-\frac{\theta}{2}}\mathscr{C}_0^{1/2}{\left| {t}^{\prime } - t\right| }^{\frac14},
\end{aligned}\]
if we take \({Jh} = O( {| {t}^{\prime } - t| }^{1/2})\), this proves \eqref{local-06}. Finally, \eqref{local-07} follows from \eqref{eq:lnv-jump0} and \eqref{local-02}. It completes the proof of Lemma \ref{thm-exsit2}.
\end{proof}

In the following lemma we derive an estimate for the quantity \(\mathscr{F}\). This higher-order regularity will be required later when we examine the pointwise behavior of the jumps \(\left\lbrack  \frac{\epsilon {u}_{x}}{v}\right\rbrack\).

\begin{lemma}\label{thm-exsit3}
Let \(\theta, \underline{v}, \hat{v}\) and \(M_{0}\) be as in Lemma \ref{thm-exsit2} and let \((v_{j},u_{k})\) be a solution of \eqref{eq:scheme} up to some time \(t_0\leq1\), with \(\mathscr{C}_{0}\leq M_{0}\), that satisfies \eqref{local-01}-\eqref{local-02} with constant \(C=C(\theta, \underline{v}, \hat{v}, M_{0})\). Then the bounds
\begin{align}
\displaystyle  &\mathscr{F}(t_0)\leq C\mathscr{C}_0, \label{local-11}\\
&\|\delta u(t)\|_{\infty}\leq C\mathscr{C}_0^{1/2}t^{-\frac{1}{2}-\frac{\theta}{4}},\label{local-12}\\
&\left\langle\frac{\epsilon\delta u_{j}}{v_{j}}-p_{j}\right\rangle_{\{t\geq \tau\}}^{1/2,1/4}\leq C\mathscr{C}_0^{1/2}\tau^{-\frac{3}{4}-\frac{\theta}{2}},\label{local-13}
\end{align}
hold for a new constant \(C=C(\theta, \underline{v}, \hat{v}, M_{0})\).
\end{lemma}

\begin{proof}
We differentiate the second equation in \eqref{eq:scheme} with respect to \(t\) , multiply the derivative by \(s^{\frac32+\theta}\dot{u}_kh\), sum and integrate to obtain
\[t^{\frac{3}{2}+\theta}\sum_{k}\frac{\dot{u}_{k}(t)^{2}}{2}h=C\int_{0}^{t}\sum_{k}s^{\frac{1}{2}+\theta}\dot{u}_{k}(s)^{2}hds-\int_{0}^{t}\sum_{j}s^{\frac{3}{2}+\theta}\delta\dot{u}_{j}\frac{d}{dt}\left(\frac{\epsilon\delta u}{v}-p\right)_{j}hds.\]
Elementary estimates based on \eqref{local-02} then show that
\[t^{\frac{3}{2}+\theta}\sum_{k}\dot{u}_{k}(t)^{2}h+\int_{0}^{t}\sum_{j}s^{\frac{3}{2}+\theta}\delta\dot{u}_{j}(s)^{2}hds\leq C\mathscr{C}_0.\]
Combining this estimate with \eqref{eq:scheme} and taking the appropriate supremum, we thus find that
\[\mathscr{F}(t)\leq C\mathscr{C}_0.\]
Inequality \eqref{local-11} follows immediately.

The regularity estimates \eqref{local-12}-\eqref{local-13} can be derived from \eqref{local-11} as follows. 
First, from \eqref{eq:ux-infty} and \eqref{local-02},
\[\begin{aligned}
\|\delta u\|_{\infty}^{2}& \leq C\left(t^{-\frac{1}{2}}\mathscr{B}+t^{-1-\frac{\theta}{2}}\mathscr{B}^{1/2}\mathscr{F}^{1/2}+(1+h^{1/2})\mathscr{C}_0^{1/2}+h^{1/2}t^{-\frac34-\frac{\theta}{2}}\mathscr{F}^{1/2}\right) \\
&\leq Ct^{-1-\frac{\theta}{2}}\mathscr{C}_{0},
\end{aligned}\]
for \(t\) and \(h=O(t)\) suitably small. This proves \eqref{local-12}. 
To prove \eqref{local-13}, we define \(w_j= \frac{\epsilon\delta u_{j}}{v_{j}}-p_{j}\) and note that \(\delta w_k=\dot{u}_k\). Thus from the definition of \(\mathscr{F}\), we have
\[\sum\delta w_{k}^{2}h\leq\sum_{k}\dot{u}_{k}^{2}h\leq Ct^{-\frac{3}{2}-\theta}\mathscr{F}(t),\]
so that, for \(0<\tau\leq t_{1}\leq t_{2}\),
\[\begin{aligned}
|w_{j}(t_{2})-w_{j}(t_{1})|& \leq\frac{1}{Jh}\sum_{l=j}^{j+J-1}\left|w_{l}(t_{2})-w_{l}(t_{1})\right|h+C\tau^{-\frac{3}{4}-\frac{\theta}{2}}\sqrt{Jh\mathscr{F}(\tau)} \\
&\leq\sqrt[]{\frac{|t_{2}-t_{1}|}{Jh}}\biggl[\int_{t_{1}}^{t_{2}}\sum_{l}\dot{w}_{l}(s)^{2}hds\biggr]^{1/2}+C\tau^{-\frac{3}{4}-\frac{\theta}{2}}\sqrt{Jh{\mathcal F}(\tau)}.
\end{aligned}\]
We note that
\(\dot{w}_{l}^{2}\leq C(\delta\dot{u}_{l}^{2}+\delta u_{l}^{4}+\|p_{v}\|_{\infty}^{2}\delta u_{l}^{2})\), so that, by the definitions of $\mathscr{B}$ and $\mathscr{F}$,
\[\left|w_j(t_2)-w_j(t_1)\right|\leq C\left|t_2-t_1\right|^{\frac14}\left(\tau^{-1/2}\mathscr{B}^{1/2}+\tau^{-\frac{3}{4}-\frac{\theta}{2}}\mathscr{F}^{1/2}\right)\leq C\left|t_2-t_1\right|^{\frac14}\mathscr{C}_0^{1/2}\tau^{-\frac{3}{4}-\frac{\theta}{2}},\]
where we have taken \(Jh=O(|t_{2}-t_{1}|^{1/4})\). The estimate yields \eqref{local-13} for \(\langle w_j\rangle =\big\langle\frac{\epsilon\delta u_{j}}{v_{j}}-p_{j}\big\rangle\). This completes the proof of Lemma \ref{thm-exsit3}.
\end{proof}

\subsection{Proof of Theorem \ref{thm:loc}}
\indent\qquad
In this subsection, we apply the difference scheme \eqref{eq:scheme} and the estimates of Lemmas \ref{thm-exsit2} and \ref{thm-exsit3} to prove Theorem \ref{thm:loc}. %when \(u_0\) is in \(H^{1}(\mathbb{T})\). The fact that the constant \(C\) of Theorems \ref{thm-exsit2} and \ref{thm-exsit3} is independent of \(\|u_{0x}\|_{L^2(\mathbb{T})}\) will then enable us to ``complete'' the solution operator, thereby including the more general initial data of Theorem \ref{thm:loc}.

\noindent\textbf{Proof of Theorem \ref{thm:loc}.} We choose \(h\) small and define \(x_k=kh\) for \(k=0,\pm 1,\dots,\pm N\), and \(x_j=jh\) for \(j=\pm\frac12,\pm\frac32,\dots, \pm(N-\frac12)\), just as in Subection 2.1.  We then take
\[\begin{aligned}&u_{k}(0)=\frac{1}{h}\int_{x_{k-\frac{1}{2}}}^{x_{k+\frac{1}{2}}}u_{0}(x)dx,\\&v_j(0)=v_0(x_j).\end{aligned}\]
We also define \(x_{\pm N_*}\) to be the mesh point \(x_k\) closest to \(\pm l\). 
It is then easy to see that
\begin{equation}\label{eq:p31}
	\mathscr{C}_0\leq CC_0,%\quad\mathscr{C}_1\leq CC_1,
\end{equation}
where the \(\mathscr{C}_{0}\) is defined in \eqref{eq:c00}. Since \(v_{j}(0)\in [b+2\underline{v},\hat{v}-\underline{v}]\) for all \(j\), we may apply Lemma \ref{thm-exsit2} to conclude that for \(h\leq h_{0}\) the scheme \eqref{eq:scheme} with these initial values are solvable up to time \(t_{0}=t_{0}(\theta, \underline{v},\hat{v}, M)\) has value \(v_{j}\in [b+\underline{v},\hat{v}]\), and satisfies the properties \eqref{local-02}-\eqref{local-07}
and \eqref{local-11}-\eqref{local-13}. 
We define the piecewise linear interpolants \(u^h(\cdot,t)\) and \(v^h(\cdot,t)\) by
\begin{equation*}
	\begin{aligned}
	\displaystyle  &u^h(\cdot,t)=\frac{x-x_k}{h}u_{k+1}(t)+\frac{x_{k+1}-x}{h}u_{k}(t),~~~&&x\in [x_k,x_{k+1}],\\
	&v^h(\cdot,t)=\frac{x-x_j}{h}v_{j+1}(t)+\frac{x_{j+1}-x}{h}v_{j}(t),~~~&&x\in [x_j,x_{j+1}],~~~j \neq \pm N_*-\frac12,
	\end{aligned}
\end{equation*}
with the exception that if \(j=\pm N_*-1/2\), then
\[v^h(x,t)=\begin{cases}v_j(t),&\quad x_j\leq x\leq \pm l\\v_{j+1}(t),&\quad \pm l<x\leq x_{j+1}.\end{cases}\]
The bounds \eqref{local-04}-\eqref{local-06} and \eqref{eq:p31} imply that the function \((v^h,u^h)\) is bounded and H\"{o}lder continuous in both \(x\) and \(t\) uniformly in \(h\), for \(h\leq h_{0}\) (\(v^{h}\) is only piecewise H\"{o}lder continuous in $x$ of course). 
It follows Arzelà–Ascoli theorem that a subsequence \((v^h,u^h)\) converges to a H\"{o}lder continuous function \((v,u)\), uniformly on compact sets in \(\{t\geq 0\}\) and \(v\) is piecewise H\"{o}lder continuous in \(x\). 
Conditions \eqref{eq:loc-v1}, \eqref{eq:loc-u1} and \eqref{eq:jump-loc} then follow directly from \eqref{local-04}-\eqref{local-07} and straightforward arguments based on \eqref{local-12}-\eqref{local-13} prove \eqref{eq:loc-u2} and \eqref{eq:loc-up}.

Next we define \(A^{h},B^{h}\), and \(F^{h}\), just like \(A, B\) and \(F\) in \eqref{eq:A}-\eqref{eq:F}, but with \((v,u)\) replaced by \((v^{h},u^h)\), and with the terms involving \(\left(\frac{u_{x}}{v}\right)_{x}\) omitted.
The bounds \eqref{local-01}, \eqref{local-11} and \eqref{eq:p31} then imply that
\begin{equation}\label{eq:p33}
	A^h+B^h+F^h\leq CC_0.
\end{equation}
It is easy to see that \eqref{eq:c0-b1}, except for the term \(\left(\frac{u_{x}}{v}\right)_{x}\), follows from \eqref{eq:p33}. For example, one of the bounds in \eqref{eq:p33} is that, for \(t\) fixed,
\[\|u_x^h(\cdot,t)\|^2\leq CC_0t^{-1/2}.\]
Thus a further subsequence \(u_{x}^{h_m}(\cdot,t)\) converges weakly in \(L^{2}(\mathbb{T})\) to a function \(w\in L^{2}(\cdot,t)\), with \(\|w\|^2\leq CC_0t^{-1/2}\). But since \(u_{x}^{h_m}(\cdot,t)\to u_{x}(\cdot,t)\) in the sense of distributions, we must have that \(u_{x}(\cdot,t)=w\in L^{2}(\mathbb{T})\) and \(\|u_{x}(\cdot,t)\|^{2}\leq CC_{0}t^{-1/2}\), as required. Thus it has been established that \((v,u)\) is indeed a weak solution of \eqref{ns-lagrange}. 
The proof of Theorem \ref{thm:loc} is complete. $\hfill\square$

\section{Global existence and large-time behavior}\label{sec3}
\indent\qquad
In this section, we will give a careful derivation of the a priori estimates, which are then applied together with the local existence theorem to complete the proof of Theorem \ref{thm1}.

\subsection{A priori estimates}\label{a-priori}
\indent\qquad
In this subsection, we assume that the quantities involved in the a priori estimates possess all the regularity properties required for the analysis. This can be guaranteed by the local existence theorem in Section \ref{sec:local}.
For simplicity, we assume the viscosity coefficient $\epsilon=1$ in the subsequent analysis. 
We define
\begin{equation}\label{eq:E0}
	E_{0}= \left\|\left(v_{0}-\tilde{v}, u_{0}-\tilde{u}, u_{0 x}\right)\right\|^{2}+\Big\|\Big(v_{0 x},\big(\frac{u_{0x}}{v_{0}}\big)_{x}\Big)\Big\|_\Sigma^{2}.
\end{equation}
We also define \(D(t_0)\) and \(G(t_0)\) to be the same as \(B(t_0)\) and \(F(t_0)\), but with \(\sigma\) taken to be identically one. For the neighborhood $D_\alpha$ and $D_\beta$ defined in \eqref{eq:dd}, there is a function $S_\alpha(v) : D_\alpha \rightarrow \mathbb{R}$ satisfying
\begin{equation}\label{s-def}
  S_\alpha(\alpha_0) = 0, \quad S_\alpha'(v) = -p(v), \quad S_\alpha''(v)= -p'(v)> 0.
\end{equation}
A similar entropy function $S_\beta$ is assumed to exist in a neighborhood $D_\beta$ of $\beta_0$. 
In the following lemma, we derive certain a priori estimates required for the analysis of the large-time behavior of the solution. 
\begin{lemma}\label{lem-apriori}
Let $D_\alpha, D_\beta$ and $(\tilde{v},\tilde{u})$ be as in Theorem \ref{thm1}, and let $S_\alpha$ and $S_\beta$ be as described above in \eqref{s-def}. Then there exist positive constants $\delta_{2} $ and $C_{2}$ such that, if a weak solution $ (v, u) $ exists on $ \left[0, t_{0}\right] $ and satisfies
\begin{align*}
&v = 
\begin{cases}
D_\alpha, & x \in I_1 \cup I_3, \\
D_\beta, & x \in I_2,
\end{cases}  \\
&A(t)+D(t)+G(t)< \infty,  \\
%\end{align*}\begin{align*}
&E_{0} \leq \delta_{2},
\end{align*}
then it holds 
\begin{equation*}
  A(t)+D(t)+G(t) \leq C_{2} E_{0}.
\end{equation*}
\end{lemma}
\begin{proof}
Throughout this proof, \( C \) denotes a generic positive constant, which is independent of \( t_{0} \). First, we compute from \eqref{ns-lagrange} and \eqref{s-def} that
\begin{align*}
\left(S_\alpha\right)_t&=-p(v)v_t=-p(v)u_x\\
&=\left(-p(v)u\right)_x+p(v)_xu\\
&=\left(-p(v)u\right)_x+\left(-u_t+\big(\frac{u_{x}}{v}\big)_{x}\right)u\\
&=-\left(\big(p-\frac{u_{x}}{v}\big)u\right)_x-\left(\frac{u^2}{2}\right)_t-\frac{u_{x}^2}{v}.
\end{align*}
Integrating over the set \( [0, t] \times(I_1\cup I_3) \), we obtain
 \begin{align*}
 \int_{I_1\cup I_3} S_\alpha(x, \cdot) d x\Big|_0^{t}=&-\int_{I_1\cup I_3} \frac{u^2}{2}d x\Big|_0^{t}-\int_0^t \int_{I_1\cup I_3}\frac{u_{x}^2}{v}d x d s\\&-\int_0^t\left(\big(p-\frac{u_{x}}{v}\big)u\right)((-l)-, s)d s+\int_0^t\left(\big(p-\frac{u_{x}}{v}\big)u\right)(l+, s)d s.
\end{align*}
Similarly, it holds for \( S_\beta \) that
\begin{align*}
\left(S_\beta\right)_t=-\left(\big(p-\frac{u_{x}}{v}\big)u\right)_x-\left(\frac{u^2}{2}\right)_t-\frac{u_{x}^2}{v}.
\end{align*}
Integrating over the set \( [0, t] \times I_2 \), we obtain
 \begin{align*}
 \int_{I_2} S_\beta(x, \cdot) d x\Big|_0^{t}=&-\int_{I_2} \frac{u^2}{2}d x\Big|_0^{t}-\int_0^t \int_{I_2}\frac{u_{x}^2}{v}d x d s\\&-\int_0^t\left(\big(p-\frac{u_{x}}{v}\big)u\right)(l-, s)d s+\int_0^t\left(\big(p-\frac{u_{x}}{v}\big)u\right)((-l)+, s)d s.
 \end{align*}
Adding and applying the jump conditions \eqref{jump0}, we obtain 
 \begin{equation}\label{slr1}
\begin{aligned}
 &\int_{I_1\cup I_3} S_\alpha(x, \cdot) d x\Big|_0^{t}+\int_{I_2} S_\beta(x, \cdot) d x\Big|_0^{t}%\\&
 +\int\frac{u^2}{2} d x\Big|_0^{t}+\int_0^t \int \frac{u_{x}^2}{v} d x d s=0.
 \end{aligned}   
 \end{equation}  
On the other hand, \eqref{s-def} shows that, for \( v \in D_{\alpha} \),
\[
S_\alpha\geq -p(\alpha_0)(v-\alpha_0)+C^{-1}(v-\alpha_0)^2.
\]
Similarly, it holds for \( S_{\beta} \) in \( D_{\beta} \) that
\[
S_\beta\geq -p(\beta_0)(v-\beta_0)+C^{-1}(v-\beta_0)^2.
\]
Applying the properties \eqref{steady-sol} of the steady-state solution \( (\tilde{v}, \tilde{u}) \), we then obtain
\begin{equation}\label{eq:v-l2}
\begin{aligned}
\displaystyle & \int_{I_1\cup I_3}C^{-1}(v-\alpha_0)^2dx+\int_{I_2}C^{-1}(v-\beta_0)^2dx\\
&\leq \tilde{p}\int_{I_1\cup I_3}(v-\alpha_0)dx+\tilde{p}\int_{I_2}(v-\beta_0)dx+\int_{I_1\cup I_3} S_\alpha(x, t) d x+\int_{I_2} S_\beta(x,t) d x\\
&\leq \tilde{p}\int_{-L}^{L}(v-\tilde{v})dx+\int_{I_1\cup I_3} S_\alpha(x, t) d x+\int_{I_2} S_\beta(x,t) d x\\
&\leq \int_{I_1\cup I_3} S_\alpha(x, t) d x+\int_{I_2} S_\beta(x,t) d x.
\end{aligned}  
\end{equation}
Thus we obtain from \eqref{slr1} and \eqref{eq:v-l2} that
 \begin{equation}\label{slr2}
\begin{aligned}
& \int\left(C^{-1}(v-\tilde{v})^2+\frac{u^2}{2}\right)d x+\int_0^t \int \frac{u_{x}^2}{v} d x d s \leq CC_0,
 \end{aligned}   
 \end{equation} 
then it shows that
 \begin{equation}\label{slr3}
\int\left((v-\tilde{v})^2+u^2\right)d x+\int_0^{t_0} \int u_{x}^2  d x d t\leq CC_0. 
 \end{equation}

Next we estimate the term \( \|v_{x}\|_\Sigma \) appearing in the definition of \( A \). By \eqref{ns-lagrange}, it has
\begin{equation}\label{v-x1}
\big(\frac{v_x}{v}\big)_t=\big(\ln v\big)_{xt}=\big(\frac{v_t}{v}\big)_x=\big(\frac{u_x}{v}\big)_x=u_t+p(v)_x.
\end{equation}
Applying the chain rule to compute \( p(v)_x \), we find that
\[
p(v)_{x}=-p'(v)v_x=-\alpha(v) \frac{v_x}{v},
\]
where \( \alpha(v) \) is a smooth function of \( v \) and away from the line \( x=\pm l \),
\begin{equation}\label{v-x2}
  \alpha(v)=-p'(v)v \geq C^{-1}.
\end{equation}
We substitute this into \eqref{v-x1},  multiply it by $\displaystyle\frac{v_x}{v}$, and integrate with respect to $x$ to obtain
\begin{align*}
\frac{1}{2}\frac{d}{dt} \int_{\Sigma} \big(\frac{v_x}{v}\big)^2d x+\int_{\Sigma} \alpha \big(\frac{v_x}{v}\big)^2 d x&=\int_{\Sigma} \frac{v_x}{v} u_t  d x.
\end{align*}
Applying the Cauchy-Schwarz inequality and integrating with respect to $t$, we then conclude that
\begin{equation}\label{v-x3}
\begin{aligned}
\sup_{0 \leq t \leq t_0}\|v_x\|_{L^2_\Sigma}^2+ \int_0^{t_0}\int_{\Sigma}  v_x^2 dxdt\leq C\left(C_0 +\int_0^{t_0} \int_{\Sigma} 
 u_t^2 d x d t\right).
\end{aligned}
\end{equation}
Defining
\begin{equation}\label{v-x4}
D_1\triangleq\sup_{0 \leq t \leq t_0} \|v_x\|_{L^2_\Sigma}^2 +\int_0^{t_0}\int_{\Sigma}  v_x^2 dxdt.
\end{equation}
From \eqref{v-x3}, we thus arrive at
$$D_1\leq  C\left[C_0+C_0D_1 +\int_0^{t_0} \int_{\Sigma} 
u_t^2d x d t\right],$$
so that, if $C_0$ is small, this yields
\begin{equation}\label{v-x6}
D_1\leq  C\left[C_0+\int_0^{t_0} \int_{\Sigma} 
u_t^2d x d t\right].
\end{equation}

Next, we derive an estimate for \( \left\|u_{x}(\cdot, t)\right\| \). Multiplying the second equation in \eqref{ns-lagrange} by \( u_{t} \) and integrating, one obtains
\begin{align*}
\displaystyle  \int u_t^2 d x&=\int_{I_1\cup I_2\cup I_3}\left(-(p-\tilde p)_x+\left(\frac{u_x}{v}\right)_x\right) u_{t} d x\\
&=\left[(p-\frac{u_x}{v})u_t-\tilde p u_{t}\right](l,t)+\left[(p-\frac{u_x}{v})u_t-\tilde p u_{t}\right](-l,t)+\int u_{x t}\left((p-\bar p)-\frac{u_x}{v}\right) d x\\
&=\int u_{x t}\left((p-\bar p)-\frac{u_x}{v}\right) d x,
\end{align*}
which we have used the jump condition \eqref{jump0}. Integrating above equality by parts in time and rearranging, we get
\begin{align*}
\int \frac{u_x^2}{2 v}\Big|_0^td x+\int_0^t \int u_t^2 d x d s
&=\int u_x(p-\bar p)\Big|_0^t d x-\int_0^t \int u_x p_t d x d s-\int_0^t \int \frac{u_x^2v_t}{2 v^2}  d x d s.
\end{align*}
Applying the estimate \eqref{slr3}, this finally leads to
\begin{equation}\label{u-x1}
\begin{aligned}
\sup_{0 \leq t \leq t_0}\|u_x\|^2+ \int_0^{t_0}\int u_t^2dxdt\leq C\left[C_0 +\int_0^{t_0} \int|u_x|^3d x d t\right].
\end{aligned}
\end{equation}
However, from the second equation in \eqref{ns-lagrange} and using \eqref{slr3}, \eqref{v-x3}, \eqref{v-x6}, have that
\begin{align*}
\int_0^{t_0}\int_{\Sigma}\left[\left(\frac{u_x}{v}\right)_x\right]^2 d x d t &\leq C \int_0^{t_0}\int_{\Sigma}\left(u_t^2+v_x^2\right) d x d t\leq C\left[C_0+\int_0^{t_0}\int_{\Sigma} u_t^2 d x d t\right].
\end{align*}
Defining
\[
D_2\triangleq\sup_{0 \leq t \leq t_0} \|u_x\|^2 +\int_0^{t_0} \|u_t\|^2 d t+\int_0^{t_0} \int_{\Sigma}\left[\left(\frac{u_x}{v}\right)_x\right]^2 d x d t,
\]
then
\begin{equation}\label{u-x2}
\begin{aligned}
D_2\leq C\left[C_0+\int_0^{t_0} \int\left(u_x^3+\left(\frac{\chi_x^2}{v^2}\right)_x^2\right) d x d t\right].
\end{aligned}
\end{equation}
Using elementary Sobolev inequality, together with our assumed bounds for \( v \), we have
\[
\left\|u_x(\cdot, t)\right\|_{\infty} \leq C\left(\int u_x^2 d x\right)^{1/4}\left(\int_\Sigma\left[\left(\frac{\epsilon u_x}{v}\right)_x\right]^2 d x\right)^{1/4},
\]
so that
\begin{equation}\label{u-x3}
\begin{aligned}
\int_{0}^{t_{0}} \int\left|u_{x}\right|^{3} d x d t &\leq C \int_{0}^{t_{0}}\left(\int u_{x}^{2} d x\right)^{5/4}\left(\int_\Sigma\left[\left(\frac{\epsilon u_{x}}{v}\right)_{x}\right]^{2} d x\right)^{1/4} d t \\
&\leq C \sup_{0 \leq t \leq t_{0}}\|u_{x}\|_{L^2}\left(\int_{0}^{t_{0}} \|u_{x}\|_{l^2}^{2}d t\right)^{3/4}\left(\int_{0}^{t_{0}} \int_\Sigma\left[\left(\frac{\epsilon u_{x}}{v}\right)_{x}\right]^{2} d x d t\right)^{1/4} \\
&\leq C C_{0}^{3/4} D_{2}^{3/4}.
\end{aligned}
\end{equation}
Combining \eqref{u-x2}, \eqref{u-x3} and \eqref{v-x6}, it shows that
\begin{equation}
  D_2\leq CC_0,\quad D_1\leq CC_0,\quad 0\leq t\leq t_0.
\end{equation}
From above, we  prove that
\begin{equation}\label{ad-1}
  A(t)+D(t)\leq CC_0,\quad 0\leq t\leq t_0.
\end{equation}

Finally, we estimate the functional $G(t)$. The second equation in \eqref{ns-lagrange} shows that
$$u_{tt}=\big(\frac{u_{x}}{v}\big)_{xt}-p(v)_{xt}.$$
Multiplying both sides by $u_t$ and integrating with respect to $x$, we obtain
\[
\frac{d}{dt}\int \frac{u_t^2}{2} d x=\int u_{x t}\left(p(v)_t-\big(\frac{u_x}{v}\big)_t\right) d x,
\]
where we have again used the jump condition \eqref{jump0}. Integrating by parts in time and rearranging, we get
\begin{equation}\label{u-t1}
 \begin{aligned}
\int \frac{u_t^2}{2}\Big|_0^{t}d x+\int_0^t \int \frac{u_{xt}^2}{v} d x d s
&=\int_0^t \int u_{xt} p_t d x d s+\int_0^t \int \frac{u_x^2u_{xt}}{v^2} d x d s.
\end{aligned} 
\end{equation}
Applying the estimate \eqref{ad-1}, we have 
\begin{equation}\label{u-t2}
 \begin{aligned}
&\int u_t^2\Big|_0^{t_0}d x+\int_0^{t_0} \int u_{xt}^2 d x d t\\
&\leq C\int_0^{t_0} \int v_t^2 d x d s+C\int_0^{t_0} \int u_x^4 d x d t\\
&\leq C\int_0^{t_0} \int u_x^2 d x d s+C\int_0^{t_0}\|u_x\|^2\|u_x\|_{L^\infty}^2 d t\\
&\leq CC_0+C\sup_{0 \leq t \leq t_0}\|u_x\|^2\int_0^{t_0}\|u_x\|\|\big(\frac{u_x}{v}\big)_x\|_\Sigma d t\\
&\leq CC_0.
\end{aligned} 
\end{equation}
On the other hand, f the second equation in \eqref{ns-lagrange} gives
\begin{equation}\label{u-t3}
\begin{aligned}
\int_\Sigma\big(\frac{u_{x}}{v}\big)_{x}^2dx
&\leq C\int_\Sigma\left(u_t^2+v_x^2\right)dx\leq C(\int u_t^2dx+C_0).
\end{aligned}
\end{equation}
Combining \eqref{u-t2} and \eqref{u-t3}, we conclude that
\begin{equation}\label{ad-2}
  G(t)\leq CC_0,\quad 0\leq t\leq t_0.
\end{equation}
This completes the proof of Lemma \ref{lem-apriori}.
\end{proof}

\subsection{Proof of Theorem \ref{thm1}}
\indent\qquad
In this subsection, we complete the proof of Theorem \ref{thm1} by applying the local existence result of Theorem \ref{thm:loc} and the a priori estimates of Lemma \ref{a-priori}.\\

\noindent\textbf{Proof of Theorem \ref{thm1}.}~
First, we take \({M} = 1\) in Theorem \ref{thm:loc} and we fix the resulting constants \({\tau }_{1}\) and \({C}_{1}\), then Theorem \ref{thm:loc}  applies to show that there is a solution \((v,u)\) defined on $[0,2\tau_1]$, which satisfies
\begin{align}
&v\in [b+\underline{v},\bar{v}] ,\quad 0\leq t\leq 2\tau_1, \label{eq:3.1}\\
&A(t)  + B(t)  + F(t)  \leq  {C}_{1}C_0 ,\quad\;0 \leq  t \leq  2{\tau }_{1}.\label{eq:3.2}
\end{align}
Without loss of generality, we let \({\tau }_{1} \leq  1\) and \({C}_{1} \geq  1\).
Equation \eqref{eq:3.2} implies that
\begin{equation}\label{eq:3.3}
	\|(v-\tilde{v})(\cdot,t)\|_{H^1_\Sigma}\leq C_1\tau_1^{-1}C_0,\quad \|u(\cdot,t)\|_{H^1}\leq C_1\tau_1^{-1}C_0,\quad 0\leq t\leq 2\tau_1,
\end{equation}
which shows that, if \({\delta }_{0}\) is sufficiently small,
\begin{equation}\label{eq:3.4}
	v \in 
\begin{cases}
D_\alpha, & x \in I_1 \cup I_3, \\
D_\beta, & x \in I_2,
\end{cases}  \quad \tau_1\leq t\leq 2\tau_1.
\end{equation}
We estimate the quantity \({E}_{0}\) appearing in Lemma \ref{a-priori}, taking \(t = {\tau }_{1}\) as the new initial time. Equation \eqref{eq:3.2} also implies that
\[
\|\left(\frac{u_x}{v}\right)_x(\cdot, t)\|_{L^2_\Sigma}^2\leq  {\tau }_{1}^{-3}F\left( {\tau }_{1}\right)  \leq  {C}_{1}{\tau }_{1}^{-3}C_0.
\]
Combining this estimate with \eqref{eq:3.3}, we then obtain that, at time \(t = {\tau }_{1}\) ,
\begin{equation}\label{eq:3.5}
	\|(v-\tilde{v},u,u_x)(\cdot,t)\|^2+\|\big(v_x,\big(\frac{u_x}{v}\big)_x\big)(\cdot, t)\|_{L^2_\Sigma}^2 \leq  {C}_{3}{\tau }_{1}^{-3}C_0,
\end{equation}
where the constant \({C}_{3}\) is a multiple of \({C}_{1}\) , with \({C}_{3} \geq  {C}_{1}\).
Now define \({\tau }_{k} = k{\tau }_{1},~k = 1,2,\ldots\). Theorem \ref{thm:loc} therefore applies at initial time \({\tau }_{1}\), proving that the solution exists up to time \({\tau }_{3}\) and satisfies
\begin{equation}\label{eq:3.6}
{A}_{1}\left( {\tau }_{3}\right)  + {B}_{1}\left( {\tau }_{3}\right)  + {F}_{1}\left( {\tau }_{3}\right)  \leq  {C}_{1}{C}_{2}{C}_{3}{\tau }_{1}^{-3}C_0.
\end{equation}
Here \({A}_{1}, {B}_{1}\), and \({F}_{1}\) are the same as \(A, B\), and \(F\), but with initial time taken to be \(t = {\tau }_{1}\) rather than \(t = 0\).
Equation \eqref{eq:3.6} shows that
\[
\|(v-\tilde{v},u)(\cdot,t)\|_{H^1_\Sigma}\leq C_1{C}_{2}{C}_{3}{\tau }_{1}^{-4}C_0,\quad \tau_1\leq t\leq \tau_3,
\]
so that, if \({\delta }_{0}\) is sufficiently small, we have
\begin{equation}\label{eq:3.4-1}
	v \in 
\begin{cases}
D_\alpha, & x \in I_1 \cup I_3, \\
D_\beta, & x \in I_2,
\end{cases} 
\end{equation}
for \({\tau }_{2} \leq  t \leq  {\tau }_{3}\). 
Lemma \ref{a-priori} therefore applies that, by the bound in \eqref{eq:3.5} for \({E}_{0}\),
\begin{equation}\label{eq:3.7}
	{A}_{1} + {D}_{1} + {G}_{1} \leq  {C}_{2}{C}_{3}{\tau }_{1}^{-3}C_0,\quad{\tau }_{1} \leq  t \leq  {\tau }_{3},
\end{equation}
where \({{A}_{1},{D}_{1}}\) and \({G}_{1}\) are the same as \(A, D\) and \(G\), but with initial time \(t = {\tau }_{1}\) rather than \(t = 0\).
Equation \eqref{eq:3.7} gives that
\begin{equation}\label{eq:3.8}
	\|(v-\tilde{v}, u)(\cdot,t)\|_{H^1_\Sigma}\leq {C}_{2}{C}_{3}{\tau }_{1}^{-3}C_0,\quad \tau_1\leq t\leq \tau_3.
\end{equation}
Thus we can still take $\tau_2$ as the initial time, and by induction we prove that a solution \((v,u)\) exists for all time and satisfies the estimate \eqref{eq:e0-b1} with constant \(C = {C}_{1} + {C}_{2}{C}_{3}{\tau }_{1}^{-3}\).

Next, we only need to prove the long-time behavior \eqref{large-time}.
We let
\[
f(t)  =\int_{\mathbb{T}}{u}_{x}^{2}(x,t) dx,
\]
using \eqref{eq:e0-b1}, we get
\[
\operatorname{Var}f\Big|_1^{\infty} = \int_1^\infty|f'(t)|dt\leq 2\int_1^\infty\int_{\mathbb{T}} |u_xu_{xt}| dxdt\leq 2C\delta,
\]
This proves that \(f(t)\) has a limit as \(t \rightarrow  \infty\). Since \(f\) is integrable on \([1,\infty )\) by \eqref{eq:e0-b1}, this limit must be zero. Thus
\begin{equation}\label{eq:large-u}
\| {u}\left( {\cdot ,t}\right) {\|}_{{L}^{\infty }}^{2} \leq  2\| u\left( {\cdot ,t}\right)\| \| {u}_{x}\left( {\cdot ,t}\right) \| \leq  {2A}{(t) }^{1/2}f{(t) }^{1/2} \rightarrow  0\;\text{ as }t \rightarrow  \infty . 
\end{equation}
The argument for \(v\) is somewhat more involved. We define
$$g_i(t)=\frac12\int_{I_i}\left(\frac{v_x}{v}\right)^2dx,\quad i=1,2,3.$$
From \eqref{ns-lagrange}, we have 
$${\left(\frac{v_x}{v}\right)}_{t}={\left(\frac{v_t}{v}\right)}_{x}={\left(\frac{u_x}{v}\right)}_{x}=u_t+p(v)_x,$$
multiplying this $\left(\frac{v_x}{v}\right)_{x}$ and integrating over $I_i, i=1,2,3,$
\[
{g}_i^{\prime}(t)  = \int_{I_i}\frac{v_x}{v}{\left(\frac{v_x}{v}\right)}_{t}{dx}= \int_{I_i}\frac{v_x}{v}\left( {{u}_{t} + {p(v)}_{x}}\right) {dx} = \int_{I_i}\frac{v_x}{v}{u}_{t}dx +\int_{I_i} \frac{\gamma(v)}{2}\left(\frac{v_x}{v}\right)^2{dx},
\]
where $\gamma(v)=2vp'(v)<0$ in $D_\alpha \cup D_\beta$. Applying the mean value theorem for integrals, we may write
\[\int_{I_i} \frac{\gamma(v)}{2}\left(\frac{v_x}{v}\right)^2{dx}=\tilde{\gamma}(t)\int_{I_i} \frac{1}{2}\left(\frac{v_x}{v}\right)^2{dx}=\tilde{\gamma}(t)g_i(t),
\]
for a different function \(\tilde{\gamma}(t)\), which satisfies \(\tilde{\gamma}(t)\leq   - 1/C\) for \(t \geq  1\) ( \(C\) now denotes a generic positive constant). Then
\[
{g}_i^{\prime}(t)  = {\tilde{\gamma} g_i} + \int_{I_i}\frac{v_x}{v}{u}_{t}dx.
\]
Solving this first-order ordinary differential equation, we obtain
\begin{align*}
\displaystyle  g_i(t)  &= \exp \left( \int_{1}^{t}\tilde{\gamma}(s)ds\right) g_i(1)  + \int_{1}^{t}\int_{I_i}\exp \left( {\int_{s}^{t}\tilde{\gamma}(\tau) {d\tau }}\right)\frac{v_x}{v}{u}_{t}{dxds}\\&\triangleq K_1+K_2.
\end{align*}
Firstly, for $L_1$, we have
$$K_1\leq C{e}^{\left( {1 - t}\right) /C} \rightarrow  0\quad \text{ as }~~ t \rightarrow  \infty.$$
Next, we write the term $K_2$ as the sum of two integrals, one for \(1 \leq  s \leq  t/2\) , the other for \(t/2 \leq  s \leq  t\) . Applying \eqref{eq:e0-b1}, one has
\begin{align*}
\displaystyle  K_2  &= \int_{1}^{t/2}\int_{I_i}\exp \left( {\int_{s}^{t}\tilde{\gamma}(\tau) {d\tau }}\right)\frac{v_x}{v}{u}_{t}{dxds}+\int_{t/2}^{t}\int_{I_i}\exp \left( {\int_{s}^{t}\tilde{\gamma}(\tau) {d\tau }}\right)\frac{v_x}{v}{u}_{t}{dxds}\\
&\leq C\left(\int_{1}^{t/2}\int_{I_i}e^{(s-t)/C}v_x^2{dxds}\right)^{1/2}\left(\int_{1}^{t/2}\int_{I_i}u_t^2{dxds}\right)^{1/2}\\&\quad+
C\left(\int_{t/2}^{t}\int_{I_i}e^{(s-t)/C}v_x^2{dxds}\right)^{1/2}\left(\int_{t/2}^{t}\int_{I_i}u_t^2{dxds}\right)^{1/2}\\
&\leq C\left(\sup_{t}\int_{I_i}v_x^2dx\right)^{1/2}\left(\int_{1}^{t/2}e^{(s-t)/C}ds\right)^{1/2}\\&\quad+
C\left(\sup_{t}\int_{I_i}v_x^2dx\right)^{1/2}\left(\int_{t/2}^{\infty}\int_{I_i}u_t^2{dxds}\right)^{1/2}\\
&\leq C\left(e^{-t/(2C)}-e^{-(1-t)/C}\right)^{1/2}+C\left(\int_{t/2}^{\infty}\int_{I_i}u_t^2{dxds}\right)^{1/2}\rightarrow  0\quad \text{ as }~~ t \rightarrow  \infty.
\end{align*}
From the above estimates, we have 
\begin{equation}\label{eq:large-g}
g_i(t) \rightarrow  0\;\text{ as }t \rightarrow  \infty,\quad i=1,2,3. 
\end{equation}
By \eqref{ns-lagrange}$_1$ and \eqref{eq:add-b0}, it yields
\begin{align*}
\displaystyle  \int_{I_1} (v(x,t)-\alpha_0)dx&=\int_{I_1} (v_0(x)-\alpha_0)dx+\int_{I_1}u_xdx\\
&\leq C\sqrt{L-l}f(t)^{\frac12}.%\rightarrow  0\;\text{ as }t \rightarrow  \infty . 
\end{align*}
Moreover, for any $x,y\in I_1$, it holds
\begin{align*}
\displaystyle  v(y,t)-\alpha_0%=v(y)-v(x)+v(x)-\alpha_0
=\int_x^y v_x(s,t)ds+v(x,t)-\alpha_0.
\end{align*}
Integrating the above equation with respect to $x$ over $I_1$ yields
\begin{align*}
\displaystyle  |v(y,t)-\alpha_0|&=\frac{1}{L-l}\Big|\int_{I_1}\int_x^y v_x(s,t)dsdx+\int_{I_1}(v(x,t)-\alpha_0)dx\Big|\\
&\leq C(L,l)g_1(t)^{1/2}+C(L,l)f(t)^{1/2}\rightarrow  0\;\text{ as }t \rightarrow  \infty,
\end{align*}
and therefore that $(v-\tilde{v})(\cdot,t) \rightarrow  0$ in $L^\infty(I_1)$ as $t \rightarrow  \infty$. A similar arguments hold on $I_2$ and $I_3$.
So that 
\begin{equation}\label{eq:large-v}
\| (v-\tilde{v})(\cdot ,t) \|_{{L}^{\infty }}\rightarrow  0\;\text{ as }t \rightarrow  \infty .
\end{equation}
The large time behavior of $(v,u)$ follows from \eqref{eq:large-v} and \eqref{eq:large-u}.
This completes the proof of Theorem \ref{thm1}. $\hfill\square$

 \section*{Acknowledgements}
The research of Y. Chen is supported by the National Natural Science Foundation of China (No.12471207). 
The research of Q. He is supported by the National Natural Science Foundation of China (No.12371434) and the National key R \& D Program of China (No.2022YFE03040002). 
The research of D. Niu is supported by Tianyuan Fund for Mathematics of the National Natural Science Foundation of China (No. 12526429) and the Natural Science Foundation of Beijing Municipality, China (No. 1252004).
The research of Y. Peng is supported by the National Natural Science Foundation of China (Nos.12301266).

\end{document}